\documentclass[12pt]{amsart}

\usepackage{geometry}                
\usepackage{graphicx}
\usepackage{amssymb}

\newtheorem{theorem}{Theorem}[section]

\newtheorem{lemma}[theorem]{Lemma}
\newtheorem{claim}[theorem]{Claim}
\newtheorem*{claim1}{Claim}

\theoremstyle{definition}
\newtheorem{definition}[theorem]{Definition}

\newtheorem{example}[theorem]{Example}

\theoremstyle{remark}
\newtheorem{remark}[theorem]{Remark}

\numberwithin{equation}{section}

\newcommand{\cS}{\mathcal{S}}
\newcommand{\cE}{\mathcal{E}}

\newcommand{\cC}{\mathcal{C}}
\newcommand{\Hh}{\mathbb{H}}
\newcommand{\D}{\mathbb{D}}
\newcommand{\B}{\mathbb{B}}
\newcommand{\C}{\mathbb{C}}

\newcommand{\R}{\mathbb{R}}

\newcommand{\rea}{\operatorname{Re}}

\newcommand{\ima}{\operatorname{Im}}
\newcommand{\Arg}{\operatorname{Arg}}

\renewcommand{\epsilon}{\varepsilon}

\title{Backward iteration in the unit ball}
\author{Olena Ostapyuk}
\address{Kansas State University\\
138 Cardwell Hall Manhattan, KS 66505}
\email{ostapyuk@math.ksu.edu}
\thanks{The author thanks her advisor Pietro Poggi-Corradini for his help and advice, and David Drasin and Adrian Jenkins for comments and suggestions.}
\begin{document}
\bibliographystyle{plain}
\begin{abstract}
We will consider iteration of an analytic self-map $f$ of the unit ball in $\C^N$. Many facts
were established about such dynamics in the 1-dimensional case (i.e. for
self-maps of the unit disk), and we will generalize some of them in higher
dimensions. In particular, in the case when $f$ is hyperbolic or elliptic, it will be
shown that backward-iteration sequences with bounded hyperbolic step 
converge to a point on the boundary. These points will be called boundary
repelling fixed points and will possess several nice properties. At each isolated boundary repelling fixed point we will also
construct a (semi) conjugation of $f$ to an automorphism via an analytic
intertwining map. We will finish with some new examples.

\end{abstract}
\maketitle
\baselineskip=18pt
\begin{section}{\bf Introduction}
\begin{subsection}{One-dimensional case}
\begin{subsubsection}{Forward iteration}
Let $f$ be an analytic self-map of the unit disk $\D$. Denote $f_n=f^{\circ n}$ and consider the sequence of forward iterates $z_n=f_n(z_0)$. By Schwarz's lemma, $f$ is a contraction of the  pseudo-hyperbolic metric, so the sequence $d(z_n,z_{n+1})$ is decreasing, where 
\[
d(z,w):=\left|\frac{z-w}{1-\overline{w}z}\right|,\hspace{0.1 in}\forall z,w\in\D.
\]

\begin{theorem}[Denjoy-Wolff]
If $f$ is not an elliptic automorphism, then there exists a unique point $p\in\overline{\D}$ (called the Denjoy-Wolff point of $f$) such that the sequence of iterates $\{f_n\}$ converges to $p$ uniformly on compact subsets of $\D$.
\end{theorem}
Consider first the case $p\in\partial\D$. It can be shown that $f(p)=p$ and $f^{\prime}(p)=c\leq 1$ in the sense of non-tangential limits, and the point $p$ can thus be called "attracting". More geometrically, Julia's lemma holds for the point $p$, i.e.
\begin{align}
\forall R>0\hspace{0.5cm}f\left(H(p,R)\right)\subseteq H(p,cR),\label{Julia1dim}
\end{align}
where $H(p,R)$ is a horocycle at $p\in\partial\D$ of radius $R$ (see Figure \ref{fig:Julia}),
\[
H(p,R):=\left\{z\in\D:\frac{|p-z|^2}{1-|z|^2}<R\right\}.
\]
\begin{figure}[ht]
	\centering
		\includegraphics[width=0.50\textwidth]{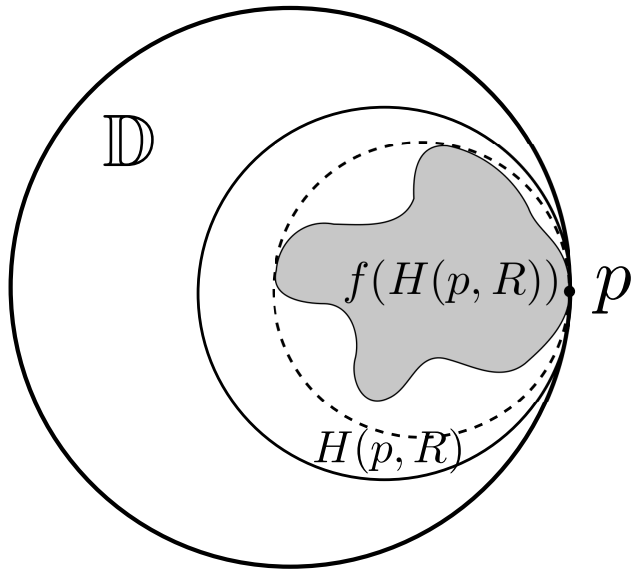}
	\caption{Julia's lemma at the Denjoy-Wolff point $p\in\partial\D$.}
	\label{fig:Julia}
\end{figure}

Here $c=f^{\prime}(p)$ is the smallest $c$ such that (\ref{Julia1dim}) holds. We will call it the {\sf multiplier} or the {\sf dilatation coefficient} and we will distinguish the {\sf hyperbolic} ($c<1$) and {\sf parabolic} ($c=1$) cases.

In the hyperbolic case, Valiron \cite{Va} showed that there is an analytic map $\psi:\D\to\Hh$ (where $\Hh$ is the right half-plane) with some regularity properties, which solves the Schr\"{o}der equation: 
\begin{align}
\psi\circ f=\frac{1}{c}\psi,\label{schforward}
\end{align}
and so $\psi$ conjugates $f$ to multiplication in $\Hh$.

In the parabolic case, $f$ can be conjugated to a shift in a half-plane or in the whole plane, as proved by Pommerenke \cite{Po}, and Baker and Pommerenke \cite{BPo}.

If the Denjoy-Wolff point $p$ is in $\D$, the function $f$ is said to be {\sf elliptic} and the multiplier $c=f^{\prime}(p)$ satisfies $|c|<1$, unless $f$ is an elliptic automorphism. Conjugations for such maps were found by Koenigs \cite{Koe} and B{\"o}ttcher \cite{Bo}.

\end{subsubsection}
\begin{subsubsection}{Backward iteration}
\begin{definition}
We will call a sequence of points $\{z_n\}_{n=0}^{\infty}$ a {\sf backward-iteration sequence} for $f$ if $f(z_{n+1})=z_n$ for $n=0,1,2,\ldots$ . 
\end{definition}
In general, such sequences may not exist. Note that in the backward iteration case the sequence $d(z_n,z_{n+1})$ is increasing, so we will impose an upper bound on the pseudo-hyperbolic step:
\begin{align}
d(z_n,z_{n+1})\leq a, \hspace{0.1in} \forall n,\label{boundedstep}
\end{align}
for some fixed $a<1$.

This condition is nontrivial, for an example of a map that admits a backward-iteration sequence with unbounded steps, see section 2 of \cite{PPC4}.

A backward-iteration sequence satisfying (\ref{boundedstep}) must converge to a point on the boundary of $\D$:
\begin{theorem}[Poggi-Corradini, \cite{PPC2}]\label{thm:main1dim}
Suppose $f$ is an analytic map with $f(\D)\subseteq\D$ (and not an elliptic automorphism). Let $\{z_n\}_{n=0}^{\infty}$ be a backward-iteration sequence for $f$ with bounded pseudo-hyperbolic steps $d_n=d(z_n,z_{n+1})\uparrow a<1$. Then the following hold:  
\begin{enumerate}
\item{There is a point $q\in\partial\D$ such that $z_n\to q$ as $n$ tends to infinity, and $q$ is a fixed point for $f$ with a well-defined multiplier $f^{\prime}(q)=\alpha<\infty$.}
\item{When $q\neq p$, where $p$ is the Denjoy-Wolff point, then $\alpha>1$, so we can call $q$ a {\sf boundary repelling fixed point}. If $q=p$, then $f$ is necessarily of parabolic type.}
\item{When $q\neq p$, then the sequence $z_n$ tends to $q$ along a non-tangential direction.}
\item{When, in the parabolic case, $q=p$, then $z_n$ tends to $q$ tangentially.}
\end{enumerate}
\end{theorem}
In this case Julia's lemma holds for the point $q$ with multiplier $\alpha>1$:
\begin{align}
\forall R>0\hspace{0.5cm}f\left(H(q,R)\right)\subseteq H(q,\alpha R),\label{Juliaalpha}
\end{align}
where $\alpha$ is the smallest number such that this holds.

For backward iteration, the following conjugation result was obtained in \cite{PPC1}:
\begin{theorem}[Poggi-Corradini]\label{thm:conj1dim}
Suppose $f$ is an analytic self-map of the unit disc $\D$ and 1 is a boundary repelling fixed point for $f$ with multiplier $1<\alpha<\infty$. Let $a=(\alpha-1)/(\alpha+1)$ and $\eta(z)=(z-a)/(1-az)$. Then there is an analytic map $\psi$ of $\D$ with $\psi(\D)\subseteq\D$, which has non-tangential limit 1 at 1, such that
\begin{align}\label{conjugation1dim}
\psi\circ\eta(z)=f\circ\psi(z),
\end{align}
for all $z\in\D$.
\end{theorem}
\end{subsubsection}
\end{subsection}
\begin{subsection}{Unit ball in $\C^N$.}
\begin{subsubsection}{Preliminaries.}
Consider the N-dimensional unit ball $\B^N=\left\{Z\in\C^N: \|Z\|<1\right\}$,\linebreak where the inner product and the norm are defined as 
\[
(Z,W)={\displaystyle\sum^{N}_{j=1}}Z_j \overline{W_j}\hspace{0.5cm}\mbox{and}\hspace{0.5cm}\|Z\|^2=(Z,Z).
\]
Schwarz's lemma still holds for a self-map $f$ of the unit ball, i.e. $f$ must be a contraction in the Bergmann metric $k_{\B^N}$ (Corollary (2.2.18) from \cite{Abate}). For simplicity of computations, we will use the pseudo-hyperbolic metric $d_{\B^N}$ in $\B^N$, which is related to the Bergmann metric by
\[
d_{\B^N}(Z,W)=\tanh(k_{\B^N}(Z,W))\hspace{0.5 cm}\forall Z,W\in\B^N.
\]
The pseudo-hyperbolic metric satisfies $d_{\B^N}(Z,0)=\left\|Z\right\|$ and is preserved by every automorphism of $\B^N$, thus one can derive that
\begin{align}
d^2_{\B^N}(Z,W)=1-\frac{(1-\left\|Z\right\|^2)(1-\left\|W\right\|^2)}{\left|1-(Z,W)\right|^2}\hspace{0.5cm}\forall Z,W\in\B^N.\label{dBN}
\end{align}
We also have the following generalization of Julia's lemma:
\begin{theorem}[Theorem (2.2.21) from \cite{Abate}]\label{thm:JuliaN}
Let $f:\B^N\to\B^N$ be a holomorphic map and take $X\in\partial\B^N$ such that
\begin{align}
\liminf_{Z\to X}\frac{1-\|f(Z)\|}{1-\|Z\|}=\alpha<\infty.\label{multalpha}
\end{align}
Then there exists a unique $Y\in\partial\B^N$ such that 
\[
\forall R>0\hspace{0.5cm}f\left(H(X,R)\right)\subseteq H(Y,\alpha R),
\]
where $H(X,R)$ is a horosphere (the N-dimensional generalization of a horocycle), defined as
\[
H(X,R):=\left\{Z\in\B^N:\frac{|1-(Z,X)|^2}{1-\|Z\|^2}<R\right\}.
\]
\end{theorem}
And a version of the Denjoy-Wolff theorem also holds:
\begin{theorem}[MacCluer, \cite{McCl}]\label{thm:DWN}
Let $f:\B^N\to\B^N$ be a holomorphic map without fixed points in $\B^N$. Then the sequence of iterates $\{f_n\}$ converges uniformly on compact subsets of $\B^N$ to the constant map $Z\mapsto p$ for a (unique) point $p\in\partial\B^N$ (called the Denjoy-Wolff point of $f$); and the number
\begin{align}
c:=\liminf_{Z\to p}\frac{1-\|f(Z)\|}{1-\|Z\|}\in(0,1]\label{multiplier}
\end{align}
is called the multiplier or the boundary dilatation coefficient of $f$ at $p$.
\end{theorem}
The map $f$ is called {\sf hyperbolic} if $c<1$ and {\sf parabolic} if $c=1$.

Unlike in the one-dimensional case, there may be many fixed points inside the unit ball $\B^N$. Even if the fixed point is unique, forward iterates need not converge to it (consider rotations). We will call a function $f$ {\sf unitary on a slice} if there exist $\zeta$ and $\eta$ in $\partial\B^N$ with $f(\lambda\zeta)=\lambda\eta$ for all $\lambda\in\D$. Functions that are not unitary on any slice are precisely those for which strict inequality occurs in the multidimensional Schwarz lemma and for them forward iterates converge to $0$ (see \cite{CowMacC}).

\begin{definition}
We will call a self-map of the unit ball $f$ {\sf elliptic}, if it has a unique fixed point inside $\B^N$ and it is conjugate via an automorphism to a self-map fixing zero, which is not unitary on any slice.
\end{definition}

In the rest of the paper we will consider only self-maps of the ball that are elliptic, hyperbolic or parabolic.

Sometimes it will be more convenient to use the Siegel domain:
\[
\Hh^N:=\left\{(z,w)\in \C\times\C^{N-1} : \rea z>\|w\|^2\right\},
\]
which is biholomorphic to $\B^N$ via the Cayley transform $\cC: \B^N\to\mathbb{H}^N$:
\[
\cC(z,w)=\left(\frac{1+z}{1-z},\frac{w}{1-z}\right)\ \ \ \mbox{and}\ \ \  \cC^{-1}(z,w)=\left(\frac{z-1}{z+1},\frac{2w}{z+1}\right).
\]

We will use the same notations for the points in $\B^N$ and their images in $\Hh^N$, when this is not likely to cause confusion. We will also denote by $(z,w)$ an $N$-dimensional vector either in $\B^N$ or $\Hh^N$ with $z\in\C$ being the first component and $w\in\C^{N-1}$ being the last $N-1$ components. The pseudo-hyperbolic distance in $\Hh^N$ is defined as
\begin{align}
d^2_{\Hh^N}((z,w),(\tilde{z},\tilde{w})):&=d^2_{\B^N}(\cC^{-1}(z,w),\cC^{-1}(\tilde{z},\tilde{w}))\notag\\&=1-\frac{4(\rea  z-\|w\|^2)(\rea  \tilde{z}-\|\tilde{w}\|^2)}{|z+\bar{\tilde{z}}-2\left\langle w,\tilde{w}\right\rangle|^2}\hspace{0.5cm}\forall (z,w),(\tilde{z},\tilde{w})\in\Hh^N.\label{dHN}
\end{align}
Forward iteration in the unit ball of $\C^N$ in the hyperbolic case was studied in \cite{BG} and \cite{PPC3}. In \cite{PPC3} the Schr\"{o}der equation (\ref{schforward}) was solved with $\psi$ being holomorphic map $\psi:\B^N\to\Hh$ given some additional conditions. In \cite{BG}, $f$ was conjugated to its linear part, assuming some regularity at the Denjoy-Wolff point. Conjugations for elliptic maps were given in \cite{CowMacC}. There are no known results for conjugations of parabolic maps in higher dimensions.
\end{subsubsection}

\begin{subsubsection}{Main results.}
The main goal of this paper is to study backward iterates in the unit ball $\B^N$. The following results are generalizations of Theorem \ref{thm:main1dim} and Theorem \ref{thm:conj1dim} to higher dimensions.

\begin{theorem}\label{thm:main}
Let $f$  be a holomorphic self-map of $\B^N$ of hyperbolic or elliptic type with Denjoy-Wolff point $p$. Let $\{Z_n\}$ be a backward-iteration sequence for $f$ with bounded pseudo-hyperbolic step $d_{\B^N}(Z_n,Z_{n+1})\leq a<1$. Then:
\begin{enumerate}
	\item{There exists a point $q\in\partial\B^N$, $q\neq p$, such that $Z_n\to q$ as $n$ tends to infinity,}
	\item{$\{Z_n\}$ stays in a Koranyi region with vertex $q$,}
	\item{Julia's lemma (\ref{Juliaalpha}) holds for $q$ with a finite multiplier $\alpha\geq \frac{1}{c}$, where $c<1$ is a constant that depends on $f$.}
\end{enumerate}
\end{theorem}

\begin{remark}
In the hyperbolic case, $c$ is the multiplier at $p$,  see (\ref{multiplier}).
\end{remark}

Because of the last statement of the Theorem (\ref{thm:main}), the multiplier $\alpha>1$, and thus we can introduce the following

\begin{definition}\label{def:brfp}
The point $q\in\partial\B^N$ is called a boundary repelling fixed point (BRFP) for $f$, if (\ref{Juliaalpha}) holds for some $\alpha>1$.
\end{definition}

\begin{remark}
It follows from Julia's lemma (Theorem \ref{thm:JuliaN}) that the above definition of multiplier is equivalent to (\ref{multalpha}).
\end{remark}

\begin{remark}
It follows from (\ref{Juliaalpha}) that $q$ also is a boundary fixed point with respect to $K$-limits and, consequently, non-tangential limits (see the proof of Theorem (2.2.29) in \cite{Abate}).
\end{remark}

\begin{definition}\label{def:Koranyi}
The Koranyi region $K(q,M)$ of vertex $q\in\partial\B^N$ and amplitude $M>1$ is the set
\begin{align}
K(q,M)=\left\{Z\in\B^N:\frac{|1-(Z,q)|}{1-\|Z\|}<M\right\}.
\end{align}
\end{definition}

Koranyi regions are natural generalizations of the Stolz regions in $\D$ and can be used to define $K$-limits:

\begin{definition}
We will say that function $f$ has K-limit $\lambda$ at $q\in\partial\B^N$ if for any $M>1$ $f(Z)\to\lambda$ as $Z\to q$ within $K(q,M)$.
\end{definition}

In one dimension this is exactly the non-tangential limit, while when $N>1$ the approach is restricted to be non-tangential only in the the radial dimension, see \cite{Abate}.

\begin{theorem}\label{thm:mainconjN}
Suppose $f$ is an analytic function of $\Hh^N$ with $f(\Hh^N)\subseteq\Hh^N$ and $0$ is a boundary repelling fixed point for $f$ with multiplier $1<\alpha<\infty$, isolated from other boundary repelling fixed points with multipliers less or equal to $\alpha$. Consider the automorphism of $\Hh^N$: $\eta(z,w)=(\alpha z,\sqrt{\alpha}w)$. Then there is an analytic map $\psi$ of $\Hh^N$ with $\psi(\Hh^N)\subseteq\Hh^N$ and $\psi(z,w)=\psi(z,0)$, which has restricted $K$-limit $0$ at $0$ (see Definition \ref{def:restrKlimit}), such that
\begin{align}
\psi\circ\eta(Z)=f\circ\psi(Z), \label{conjugation}
\end{align}
for every $Z\in\Hh^N$. 
\end{theorem}

It follows from the proof of Theorem \ref{thm:mainconjN} (see Lemma \ref{lemma:mainlemma}), that every isolated boundary repelling fixed point is a limit of some backward-iteration sequence with bounded hyperbolic step. Thus in the hyperbolic and elliptic cases we have the following characterization of BRFP in terms of backward-iteration sequences: Every backward-iteration sequence with bounded hyperbolic step converges to a BRFP; and if a BRFP is isolated, then we can construct a backward-iteration sequence with bounded hyperbolic step that converges to it.

The intertwining map $\psi$ in Theorem \ref{thm:mainconjN} satisfies $\psi(z,w)=\psi(z,0)$ and essentially is a map from one dimensional subspace of $\Hh^N$ to $\Hh^N$, therefore that conjugation does not provide information about behavior of $f$ outside of one dimensional image of $\psi$. It then is natural to identify situations in which we can find a conjugation such that the image of the intertwining map $\psi$ has larger dimension.

\begin{theorem}\label{thm:conjexpand} Let $f$ be expandable at $0$ (see Definition \ref{def:expandable}) and $0$ be a boundary repelling fixed point $0$ with multiplier $1<\alpha<\infty$. Assume further that the matrix $A$ in the definition of $f$ is diagonal, and without loss of generality let its eigenvalues be $a_{j,j}=\sqrt{\alpha}e^{i\theta_j}$ for $j=1\ldots L$ ($L$ is an integer, $0\leq L\leq N-1$) and $|a_{j,j}|^2<\alpha$ for $j=L+1\ldots N-1$. Define $\Omega$ as a diagonal matrix with $\Omega_{j,j}=e^{i\theta_j}$ for $j=1\ldots L$ and $\Omega_{j,j}=1$ for $j=L+1\ldots N-1$. Then the conjugation (\ref{conjugation}) holds for $\eta(z,w)=\left({\alpha} z,\Omega{\alpha}^{1/2}w\right)$ and intertwining map $\psi$ such that $\psi(z,w)=\psi(p_L(z,w))$, where $p_L$ is a projection on the first $L+1$ dimensions.
\end {theorem}

In the last section we will provide some new examples, in particular, functions in the two-dimensional Siegel domain that have non-isolated BRFPs, a phenomenon that never occurs in one dimension. In Example \ref{ex:quadpol}, we will show that the quadratic function $f(z,w):=(2z+w^2,w)$ is of hyperbolic type with the Denjoy-Wolff point infinity and has a curve $\{(r^2,ir)|r\in\R\}$ of boundary repelling fixed points, all of them having the same multiplier $\alpha=2$.

In Example \ref{ex:1to2dim} we will describe a non-trivial way to construct a function $f$ of the two-dimensional Siegel domain based on a function $\phi$ of a one-dimensional half-plane.  $f$ will behave very similarly to $\phi$ and will inherit many properties, however, it may have non-isolated BRFPs.

We will finish with a discussion of open questions.
\end{subsubsection}
\end{subsection}
\end{section}

\begin{section}{\bf Convergence of backward-iteration sequences}
 
\begin{proof}[Proof of Theorem \ref{thm:main} (hyperbolic case)]
We will move to the Siegel domain $\Hh^N$. Without loss of generality we can assume that the Denjoy-Wolff is infinity. Also denote backward-iteration sequence as $Z_n=(z_n,w_n)\in\C\times\C^{N-1}$ and define $t_n=\rea z_n-\|w_n\|^2$. The image of the horosphere centered at $(1,0)$ of radius $R$ under the Cayley transform will be 
\[
\left\{(z,w)\in\Hh^N :\frac{\left|1-\left(\cC^{-1}(z,w),(1,0)\right)\right|^2}{1-\left\|\cC^{-1}(z,w)\right\|^2}<R\right\},
\]
\[
\left\{(z,w)\in\Hh^N :\displaystyle\frac{\left|1-\frac{z-1}{z+1}\right|^2}{1-\left|\frac{z-1}{z+1}\right|^2-\frac{\|2w\|^2}{|z+1|^2}}<R\right\},
\]
and after some computations,
\[
\left\{(z,w)\in\Hh^N :\rea z-\|w\|^2>\frac{1}{R}\right\},
\]
i.e. any horosphere centered at the Denjoy-Wolff point $\infty$ will have form 
\[
H(t)=\left\{(z,w)\in\Hh^N \left|\right.\rea z-\|w\|^2>t\right\},
\]
for some $t>0$, and the Siegel domain version of the multi-dimensional Julia's lemma (Theorem \ref{thm:JuliaN}) at infinity will be
\[
\forall R>0\hspace{1cm}f\left(H\left(\frac{1}{R}\right)\right)\subset H\left(\frac{1}{cR}\right)
\]
or
\begin{align}
\forall t>0\hspace{1cm}f\left(H(ct)\right)\subset H(t).\label{JuliaH}
\end{align}

Since $f(Z_{n+1})=Z_n\notin H(t_n)$, by (\ref{JuliaH}) $Z_{n+1}\notin H(ct_n)$, and, by induction, $Z_{n+k}\notin H(c^k t_n)$, $k=1,2, \ldots$. Thus we have
\begin{align}
\rea z_{n+k}-\|w_{n+k}\|^2=t_{n+k}\leq c^k t_n,\hspace{0.8cm}k=1,2, \ldots\label{tngrowth}
\end{align}

Since the dilatation coefficient at the Denjoy-Wolff point $c<1$, the sequence $Z_n$ must tend to the boundary of the Siegel domain as n tends to infinity. All we need to show now is that the limiting set on the boundary is just one point.

Define a Euclidean projection on the boundary of the Siegel domain as
\[
pr(z,w):=(i\ima z+\|w\|^2,w).
\]
It will be enough to show that $pr(Z_n)$ has a limit.

\begin{lemma}\label{lemma:growth} The Euclidean distance between projections of consecutive points of the backward-iteration sequence is bounded by
\[
\left\|pr(Z_{n})-pr(Z_{n+1})\right\|\leq\tilde{C}\sqrt{t_n},
\]
for some positive constant $\tilde{C}$ independent of $n$.
\end{lemma}

Assuming lemma and using (\ref{tngrowth}), we have 
\begin{align}
\left\|pr(Z_{n})-pr(Z_{n+k})\right\|&\leq\sum_{j=0}^{k-1}\left\|pr(Z_{n+j})-pr(Z_{n+j+1})\right\|\leq
\tilde{C}\sum_{j=0}^{k-1}\sqrt{t_{n+j}}\leq\tilde{C}\sum_{j=0}^{k-1}\sqrt{c^j t_{n}}\notag\\
&\leq\tilde{C}\sqrt{t_n}\sum_{j=0}^{\infty}\sqrt{c^j}=\frac{\tilde{C}\sqrt{t_n}}{1-\sqrt{c}}   \xrightarrow[n\to\infty]{}0.\label{projgrowth}
\end{align}
Thus $\left\{pr(Z_{n})\right\}$ is a Cauchy sequence and must have a limit $q\in\partial\Hh^N$, which is also the limit for $\left\{Z_{n}\right\}$. Clearly, $q$ is finite and cannot coincide with the Denjoy-Wolff point.

\begin{proof}[Proof of Lemma \ref{lemma:growth}] Consider the images of $Z_n$ and $Z_{n+1}$ under the automorphism in $\Hh^N$ defined by
\[
h_n(z,w):=(z-i\ima z_n+\|w_n\|^2-2\left\langle w,w_n\right\rangle,w-w_n),
\]
which maps $Z_n$ to $(t_n,0)$. Denote $h_n(Z_{n+1})=\tilde{Z}_n=(\tilde{z}_n,\tilde{w}_n)=(\tilde{x}_n+i\tilde{y}_n,\tilde{w}_n)$. Note that $h_n$ is an isometry with respect to the pseudo-hyperbolic distance $d_{\Hh^N}$ (\cite{Abate}) and does not change the horoshperes centered at infinity $H(t)$, because
\[
\rea(z-i\ima z_n+\|w_n\|^2-2\left\langle w,w_n\right\rangle)-\|w-w_n\|^2=\rea z+\|w_n\|^2-2\rea \left\langle w,w_n\right\rangle-\|w-w_n\|^2
\]
\[
=\rea z+\|w_n\|^2-2\rea \left\langle w,w_n\right\rangle-\|w\|^2+2\rea \left\langle w,w_n\right\rangle-\|w_n\|^2=\rea  z-\|w\|^2.
\]
Thus $h_n$ will be called translations.

The point $(\tilde{z}_n,\tilde{w}_n)$ must satisfy two conditions (see Figure \ref{fig:Siegelrestrict}). First, $d_{\Hh^N}\left(Z_n,Z_{n+1}\right)\leq a$, which will take form
\begin{align}
\left|\frac{\tilde{z}_n-t_n}{\tilde{z}_n+t_n}\right|^2+\frac{4t_n\|\tilde{w}_n\|^2}{|\tilde{z}_n+t_n|^2}\leq a^2.\label{distH}
\end{align}
Second, by Julia's lemma (\ref{JuliaH})
\begin{align}
t_{n+1}=\rea \tilde{z}_n-\|\tilde{w}_n\|^2\leq ct_n.\label{tcond}
\end{align}
\begin{figure}[h]
	\centering
		\includegraphics[width=0.50\textwidth]{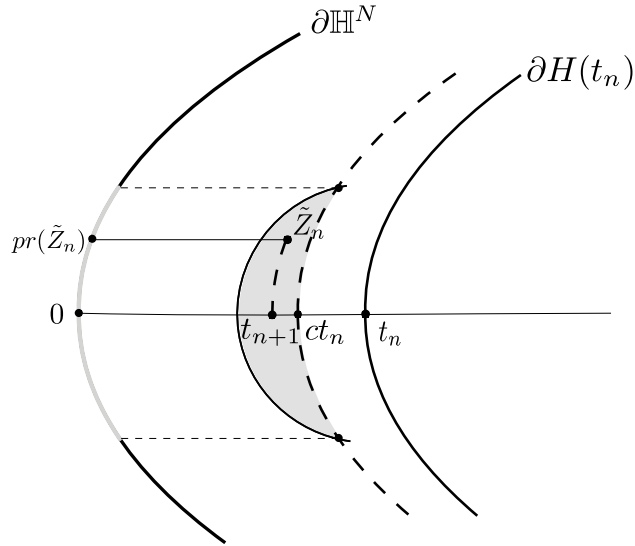}
	\caption{The restriction on the point $\tilde{Z}_n=h_n(Z_{n+1})$ and its projection on the boundary of the Siegel domain. The shaded area represents the intersection of the solutions of (\ref{distH}) and (\ref{tcond}).}
	\label{fig:Siegelrestrict}
\end{figure}

Using (\ref{distH}) and (\ref{tcond}) we obtain
\begin{align*}
|\tilde{z}_n-t_n|^2+4t_n \rea \tilde{z}_n &\leq a^2|\tilde{z}_n+t_n|^2-4t_n\|\tilde{w}_n\|^2+4t_n(ct_n+\|\tilde{w}_n\|^2),\\
|\tilde{z}_n-t_n|^2+4t_n \rea \tilde{z}_n &\leq a^2|\tilde{z}_n+t_n|^2+4ct_n^2,\\
|\tilde{z}_n+t_n|^2 &\leq a^2|\tilde{z}_n+t_n|^2+4ct_n^2,\\
|\tilde{z}_n+t_n|^2 &\leq\frac{4ct_n^2}{1-a^2},\\
|\tilde{x}_n+t_n|^2+|\tilde{y}_n|^2 &\leq\frac{4ct_n^2}{1-a^2}.
\end{align*}
Thus 
\begin{align}
&\tilde{x}_n\leq\frac{2t_n\sqrt{c}}{\sqrt{1-a^2}}-t_n=C_1 t_n,\label{xshiftgrowth}\\
&|\tilde{y}_n|\leq\frac{2t_n\sqrt{c}}{\sqrt{1-a^2}}=C_2 t_n,\label{yshiftgrowth}\\
&\|\tilde{w}_n\|^2<\tilde{x}_n\leq C_1 t_n,\label{wshiftgrowth}
\end{align}
with $C_1$ and $C_2$ independent of $n$. Note that we must have $d_{\Hh^N}(ct_n,t_n)\leq d_{\Hh^N}(\tilde{Z_n},(t_n,0))\leq a$, otherwise the backward-iteration sequence will not exist. It follows that $4c>1-a^2$ and $C_1>0$.

Now
\[
pr(Z_n)=(i \ima  z_n+\|w_n\|^2,w_n)
\]
and
\[
pr(Z_{n+1})=pr(h_n^{-1}(\tilde{z}_n,\tilde{w}_n))=\left(i \ima (\tilde{z}_n+z_n)+2 \ima \left\langle \tilde{w}_n, w_n\right\rangle+\|\tilde{w}_n+w_n\|^2, \tilde{w}_n+w_n\right).
\]
\begin{align}
pr(Z_{n+1})-pr(Z_n)=\left(i \ima \tilde{z}_n+2 \ima \left\langle \tilde{w}_n, w_n\right\rangle+\|\tilde{w}_n+w_n\|^2-\|w_n\|^2,\tilde{w}_n\right)\notag\\
=\left(i \ima \tilde{z}_n+2\left\langle \tilde{w}_n, w_n\right\rangle+\|\tilde{w}_n\|^2,\tilde{w}_n\right).\label{pr1}
\end{align}
\[
\|pr(Z_{n+1})-pr(Z_n)\|^2=\left|i \ima \tilde{z}_n+2\left\langle \tilde{w}_n, w_n\right\rangle+\|\tilde{w}_n\|^2\right|^2+\|\tilde{w}_n\|^2
\]
\[
\leq\left(|\tilde{y}_n|+2\|\tilde{w}_n\|\|w_n\|+\|\tilde{w}_n\|^2\right)^2+\|\tilde{w}_n\|^2\leq
\left(C_2 t_n+2C_1 t_n\|w_n\|+C_1 t_n\right)^2+C_1 t_n\leq\tilde{C}^2t_n,
\]
using (\ref{yshiftgrowth}), (\ref{wshiftgrowth}) and the facts that $t_n\to 0$ and assuming that $\|w_n\|$ is bounded.

Thus it is enough to show now is that $\|w_n\|\leq C_3$. Note that $w_{n+1}=w_n+\tilde{w}_n$ $\forall n$ and thus 
\begin{align*}
\|w_n\|&\leq\|\tilde{w}_{n-1}\|+\|\tilde{w}_{n-2}\|+\ldots+\|\tilde{w}_0\|+\|w_0\|\\
&\leq\sqrt{C_1}\left( \sqrt{t_{n-1}}+\sqrt{t_{n-2}}+\ldots+\sqrt{t_0}\right)+\|w_0\|\\
&\leq\sqrt{C_1}\sqrt{t_0}\left( \sqrt{c^{n-1}}+\sqrt{c^{n-2}}+\ldots+1\right)+\|w_0\|\leq\frac{\sqrt{C_1}\sqrt{t_0}}{1-\sqrt{c}}+\|w_0\|=:C_3.
\end{align*}
\end{proof}
Now we want to show that $\left\{Z_{n}\right\}$ stays in the Koranyi region with vertex $q$. Without loss of generality, take $q=0$. A Koranyi region with vertex $0$ in $\Hh^N$ must be the image under the Cayley transform of a Koranyi region with vertex $(-1,0)$ in $\B^N$, i.e. the set
\[
\left\{(z,w)\in\Hh^N:\frac{|1-(\cC^{-1}(z,w),(-1,0))|}{1-\|\cC^{-1}(z,w)\|}<M\right\}.
\]

Since $1<1+\|\cC^{-1}(z,w)\|<2$, it is enough to show that
\[
\frac{|1-(\cC^{-1}(z,w),(-1,0))|}{1-\|\cC^{-1}(z,w)\|^2}<\frac{M}{2}.
\]
The left-hand side is
\[
\frac{\left|1+\frac{z-1}{z+1}\right|}{1-\left|\frac{z-1}{z+1}\right|^2-\frac{4\|w\|^2}{|z+1|^2}}=\frac{|z+1+z-1||z+1|}{|z+1|^2-|z-1|^2-4\|w\|^2}=\frac{2|z||z+1|}{4\rea z-4\|w\|^2},
\]
thus for $Z_n=(z_n,w_n)\in\Hh^N$ we need
\[
\frac{|z_n||z_n+1|}{(\rea z_n-\|w_n\|^2)}<M.
\]
Since $|z_n+1|>1$ and bounded near $0$, and $\rea  z_n-\|w_n\|^2=t_n$, it is sufficient to show that $|z_n|\leq C t_n$ for some constant $C$ independent of $n$. Using Lemma \ref{lemma:growth}, similarly to (\ref{projgrowth}) we have
\begin{align*}
\left\|pr(Z_{n})\right\|(=\left\|pr(Z_{n})-q\right\|)=\lim_{k\to\infty}\left\|pr(Z_{n})-pr(Z_{n+k})\right\|&\leq\sum_{j=0}^{\infty}\left\|pr(Z_{n+j})-pr(Z_{n+j+1})\right\|\\
&\leq\tilde{C}\sum_{j=0}^{\infty}\sqrt{t_{n+j}}\leq\frac{\tilde{C}\sqrt{t_n}}{1-\sqrt{c}},
\end{align*}
so $\left\|pr(Z_{n})\right\|^2=\left|\ima  z_n+\|w_n\|^2\right|^2+\|w_n\|^2\leq (\frac{\tilde{C}}{1-\sqrt{c}})^2 t_n=C_4 t_n$. It follows that $\|w_n\|^2\leq C_4 t_n$. If there is a bound 
\begin{align}
\left|\ima  z_n+\|w_n\|^2\right|=|z_n-t_n|\leq C_5 t_n,\label{pr1growth}
\end{align}
then
\[
|z_n|\leq|z_n-t_n|+t_n\leq(C_5+1)t_n,
\]
and $Z_n$ must stay in the Koranyi region. It is enough to show (\ref{pr1growth}).

Denote $pr_1(Z_n)=\ima  z_n+\|w_n\|^2$, which is the first component of $pr(Z_n)$. As in (\ref{pr1})
\[
pr_1(Z_{n+1})-pr_1(Z_n)=i\tilde{y}_n+\|\tilde{w}_n\|^2+2\left\langle \tilde{w}_n, w_n\right\rangle
\]
and thus
\begin{align*}
\left|pr_1(Z_{n+1})-pr_1(Z_n)\right|&\leq|\tilde{y}_n|+\|\tilde{w}_n\|^2+2\|\tilde{w}_n\|\|w_n\|\\
&\leq C_2 t_n+C_1 t_n+2\sqrt{C_1 t_n}\sqrt{C_4 t_n}=C_6 t_n.
\end{align*}
\begin{align*}
\left|pr_1(Z_n)-0\right|=\lim_{k\to\infty}\left|pr_1(Z_n)-pr_1(Z_{n+k})\right|&\leq\sum_{k=0}^{\infty} \left|pr_1(Z_{n+k})-pr_1(Z_{n+k+1})\right|\\
&\leq C_6\sum_{k=0}^{\infty}t_{n+k}\leq C_6\sum_{k=0}^{\infty}c^k t_{n}\leq C_5 t_n,
\end{align*}
which proves (\ref{pr1growth}).

Now we will show that Julia's lemma (Theorem \ref{thm:JuliaN}) is applicable to the point $q$. Once again, assume that $q=(-1,0)$ in $\B^N$ or $q=0$ in $\Hh^N$. 
\[
\liminf_{Z\to(-1,0)}\frac{1-\left\|f(Z)\right\|}{1-\left\|Z\right\|}\leq\liminf_{n\to{\infty}}\frac{1-\left\|Z_n\right\|^2}{1-\left\|Z_{n+1}\right\|^2}.
\]
The latter liminf in $\Hh^N$ will take form
\[
\liminf_{n\to{\infty}}\frac{\rea  z_n-\|w_n\|^2}{\rea  z_{n+1}-\|w_{n+1}\|^2}\frac{|z_{n+1}+1|^2}{|z_{n}+1|^2}= \liminf_{n\to{\infty}}\frac{t_n}{t_{n+1}}.
\]

It is enough to show that $t_{n+1}\geq K t_n$ for some constant $K$. Since $d(Z_n,Z_{n+1})\leq a$, $H(t_{n+1})$ must intersect the pseudo-hyperbolic sphere (\ref{distH}), and thus
\[
\frac{t_n-t_{n+1}}{t_n+t_{n+1}}\leq a,
\]
and it follows that
\[
t_{n+1}\geq\frac{1-a}{1+a}t_n,
\]
so Julia's lemma (\ref{Juliaalpha}) holds with finite multiplier $\alpha\leq\frac{1+a}{1-a}$.

Now we will show that there is also a lower bound on $\alpha$:
\begin{align}
\alpha\geq\frac{1}{c},\label{calpha}
\end{align}
where $c<1$.

Consider the image of $0$ in $\B^N$ and denote $f(0)=(z_0,w_0)$. Since $0\in\partial H((1,0),1)$ (here $H((1,0),1)$ is a horosphere centered at the Denjoy-Wolff point $(1,0)$ of radius $1$), by Julia's lemma applied to $(1,0)$, $f(0)\in\overline{H}((1,0),c)$, where $c<1$. This horosphere is a Euclidean ellipsoid, centered at $(\frac{1}{1+c},0)$, whose restriction to the $1$-dimensional subspace, generated by $e_1=(1,0)$ is a disk of radius $\frac{c}{1+c}$ (see \cite{Abate}, (2.2.22)). Thus 
\[
\rea  z_0\geq\frac{1-c}{1+c}.
\]
In a similar way, by Julia's lemma applied to $q=(-1,0)$, $f(0)\in\overline{H}((-1,0),\alpha)$ and
\[
\rea  z_0\leq\frac{\alpha-1}{\alpha+1},
\]
so we have
\[
\frac{\alpha-1}{\alpha+1}\geq\frac{1-c}{1+c},
\]
which is equivalent to $c\alpha\geq 1$ and (\ref{calpha}) follows.
\end{proof}

\begin{proof}[Proof of Theorem \ref{thm:main} (elliptic case)]

Without loss of generality assume $0$ is the Denjoy-Wolff point. We will need the following result on the growth of function $f$ near the boundary of the ball:

\begin{lemma}\label{lemma:elliplem}
Let $f$ be a self-map of the unit ball $\B^N$ fixing zero, not unitary on any slice. Fix $r_0>0$, define $M(r):=\max\|f(r\B^N)\|$, $r\in[r_0,1)$. Then there exists $c=c(r_0)<1$ such that 
\begin{align}
\frac{1-r}{1-M(r)}\leq c\hspace{.2in} \forall r\in[r_0,1) \label{ratio}
\end{align}
\end{lemma}
\begin{proof}
Assume opposite: $\forall c<1$ $\exists z=z(c)$ with $|z|\geq r_0$ such that
\[
\frac{1-\|z\|}{1-\|f(z)\|}>c
\]
Construct the sequence $z_n:=z(\frac{n-1}{n})$. Let $z_0$ be a partial limit of $\{z_n\}$. If $z_0\in\B^N$, then $f(z_0)\in\B^N$ and
\[
\frac{1-\|z_0\|}{1-\|f(z_0)\|}\geq 1\hspace{.2in}\Leftrightarrow\hspace{.2in}1-\|z_0\|\geq 1-\|f(z_0)\|\hspace{.2in}\Leftrightarrow\hspace{.2in}\|f(z_0)\|\geq\|z_0\|,
\]
contradiction, since $z_0\neq 0$. Thus $z_0\in\partial\B^N$ and we pick a subsequence $z_{n_k}\to z_0$. Then
\[
\limsup_{k\to\infty}\frac{1-\|z_{n_k}\|}{1-\|f(z_{n_k})\|}\geq 1\hspace{.2in}\Leftrightarrow\hspace{.2in}\liminf_{k\to\infty}\frac{1-\|f(z_{n_k})\|}{1-\|z_{n_k}\|}\leq 1
\]
Applying Julia's lemma to the point $z_0\in\partial\B^N$, we obtain that $\exists w_0\in\partial\B^N$ such that $\forall R>0$ $f(H(z_0,R))\subseteq H(w_0,R)$, where $H(z,R)$ is a horosphere centered at $z$ of radius $R$.

Pick $R$ small enough such that $0\not\in \overline{H(z_0,R)}$. Let $\xi$ be a point in $\overline{H(z_0,R)}$, closest to the origin. Since $f(\xi)\in\overline{H(w_0,R)}$, we have $\|f(\xi)\|\geq\|\xi\|$ (the horospheres have the same radius). Contradiction.
\end{proof}

Denote the distance to the boundary $t_n:=1-\|Z_n\|$. By Lemma (\ref{lemma:elliplem}) we have 
\begin{align}
t_{n+k}\leq c^k t_n\hspace{.2in} \forall n,k\geq 0 \label{tgrowth},
\end{align}
where $c:=c(\|Z_0\|)$ as in Lemma (\ref{lemma:elliplem}).

Thus $t_n\leq c^n t_0\to 0$ as $n$ tends to infinity and the sequence $\left\{Z_n\right\}_{n=0}^{\infty}$ must tend to the boundary of the ball. Now denote $\phi_n$ the angle between $Z_n$ and $Z_{n+1}$ seen from the origin (which is also the arc-length between radial projections of $Z_n$ and $Z_{n+1}$ on the boundary of the ball - see Figure \ref{fig:ellipticproof}).

\begin{figure}[ht]
	\centering
		\includegraphics[width=0.50\textwidth]{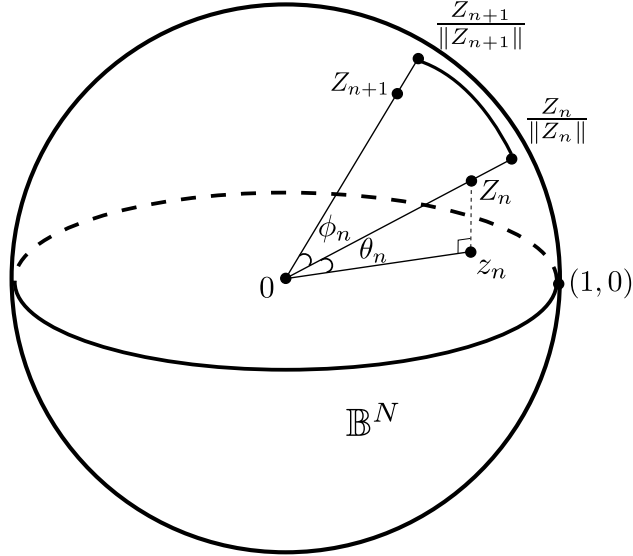}
	\caption{Two consecutive points $Z_n$ and $Z_{n+1}$ and their radial projections on the boundary of the ball.}
	\label{fig:ellipticproof}
\end{figure}

Because $d_{\B^N}(Z_n,Z_{n+1})\leq a$, $Z_{n+1}$ must be inside of the pseudo-hyperbolic ball of radius $a$ centered at $Z_n$, which is the Euclidean ellipsoid centered at $\frac{1-a^2}{1-a^2\|Z_n\|^2}Z_n$ and largest semiaxis $a\sqrt{\frac{1-\|Z_n\|^2}{1-a^2\|Z_n\|^2}}$, so as $Z_n$ tends to the boundary,
\begin{align}
\phi_n\leq C_1(1-\|Z_n\|)^{1/2}=C_1 \sqrt{t_n}. \label{phibound}
\end{align}
Then the arc-length between $\frac{Z_n}{\|Z_n\|}$ and $\frac{Z_{n+k}}{\|Z_{n+k}\|}$ does not exceed
\[
\sum_{j=0}^{k}\phi_{n+j}\leq C_1\sum_{j=0}^{k}\sqrt{t_{n+j}}\leq C_1 \sqrt{t_n}\sum_{j=0}^{k}c^{k/2}\leq C_1\frac{1}{1-\sqrt{c}}\sqrt{t_n},
\]
which tends to $0$ when $n$ tends to infinity, so the sequence of projections must converge to some point on the boundary, denote it $q$. Thus the sequence $Z_n$ must tend to $q$.

The next step is to show that $Z_n$ stays in a Koranyi region centered at $q$. Without loss of generality assume $q=(1,0)$ and denote $Z_n=(z_n,w_n)\in\C\times\C^{N-1}$. We need to show that
\begin{align}
\frac{|1-z_n|}{1-\|Z_n\|}<M \label{Koranyi}
\end{align}
for some $M>1$. By (\ref{tgrowth}) and (\ref{phibound}), The arc-length between $(1,0)$ and the projection of $Z_n$ on the boundary is bounded by
\begin{align}
\sum_{j=n}^{\infty}\phi_j\leq C_1\sum_{j=n}^{\infty}\sqrt{t_j}\leq C_2 \sqrt{t_n}.\label{angles}
\end{align}
Let $\theta_n$ be the angle between $Z_n$ and $z_n$ (i.e. the angle between $Z_n$ and the plane spanned by $(1,0)$). By (\ref{angles}), $\theta_n\leq C_2 \sqrt{t_n}$. Then
\[
1-|z_n|=1-\|Z_n\|\cos\theta_n=1-\cos\theta_n+\cos\theta_n-\|Z_n\|\cos\theta_n\leq 1-\cos\theta_n+1-\|Z_n\|\leq C_3 t_n,
\]
since $1-\cos\theta_n=\frac{\theta_n^2}{2}+o(\theta_n^3$) as $n\to\infty$.

Since $d_{\D}(z_n,z_{n+1})\leq d_{\B^N}(Z_n,Z_{n+1})\leq a$ and the pseudo-hyperbolic disk centered at $z_n$ of radius $a$ is a Euclidean disk with center $w=\frac{1-a^2}{1-a^2|z_n|^2}z_n$ and radius $r=\frac{1-|z_n|^2}{1-a^2|z_n|^2}a$,
\[
|\Arg z_n-\Arg z_{n+1}|\approx\sin|\Arg z_n-\Arg z_{n+1}|\leq\frac{r}{|w|}=\frac{a}{|z_n|}\frac{1-|z_n|^2}{1-a^2}\leq C_4 t_n
\]
Now
\[
|\Arg z_n|=|\Arg z_n-\Arg 1|\leq\sum_{k=n}^{\infty}|\Arg z_k-\Arg z_{k+1}|\leq\sum_{k=n}^{\infty}C_4 t_k\leq C_5 t_n
\] 
and
\[
|1-z_n|^2=(\ima z_n)^2+(1-\rea z_n)^2=|z_n|^2\sin^2\Arg z_n+(1-|z_n|\cos\Arg z_n)^2\leq
\]
\[
\sin^2\Arg z_n+(1-\cos\Arg z_n+1-|z_n|)^2\leq C_6 t_n^2,
\]
and (\ref{Koranyi}) follows.

For Julia's lemma to hold we need to prove that
\[
\liminf_{Z\to(1,0)}\frac{1-f(\|Z\|)}{1-\|Z\|}<\infty.
\] 
Since $\left\{Z_n\right\}_{n=0}^{\infty}$ is a backward-iteration sequence tending to $(1,0)$,
\[
\liminf_{Z\to(1,0)}\frac{1-f(\|Z\|)}{1-\|Z\|}\leq\liminf_{n\to\infty}\frac{1-\|Z_n\|}{1-\|Z_{n+1}\|},
\]
and it is enough to show that the latter liminf is finite. Note that $Z_{n+1}$ must be in the (Euclidean) ellipsoid centered at $\frac{1-a^2}{1-a^2\|Z_n\|^2}Z_n$ with radius $r=\frac{1-|Z_n|^2}{1-a^2|Z_n|^2}a$ in the subspace generated by $Z_n$, and $R=a\sqrt{\frac{1-\|Z_n\|^2}{1-a^2\|Z_n\|^2}}$ in the dimensions orthogonal to $Z_n$. Thus the point $W$, closest to the boundary, must have norm
\begin{align*}
\|W\|=\frac{1-a^2}{1-a^2\|Z_n\|^2}\|Z_n\|+\frac{1-\|Z_n\|^2}{1-a^2\|Z_n\|^2}a=\frac{(\|Z_n\|+a)(1-a\|Z_n\|)}{1-a^2\|Z_n\|^2}=\frac{\|Z_n\|+a}{1+a\|Z_n\|}
\end{align*}
and
\begin{align*}
1-\|Z_{n+1}\|\geq 1-\|W\|=1-\frac{\|Z_n\|+a}{1+a\|Z_n\|}=\frac{(1-a)(1-\|Z_n\|)}{1+a\|Z_n\|}.
\end{align*}
Thus
\begin{align*}
\frac{1-\|Z_n\|}{1-\|Z_{n+1}\|}\leq\frac{1+a\|Z_n\|}{1-a}\leq\frac{1+a}{1-a},
\end{align*}

and Julia's lemma holds with multiplier $\alpha\leq\frac{1+a}{1-a}$. The lower bound on the multiplier $\alpha\geq\frac{1}{c}$ is the direct consequence of the Lemma (\ref{lemma:elliplem}).

Note that the above results will hold for $c=c(\|Z_n\|)$ $\forall n\geq 0$, and since $\|Z_n\|\to 1$, for
\begin{align*}
c:=\lim_{r_0\to 1}c(r_0).
\end{align*}
\end{proof}

\end{section}

\begin{section}{\bf Construction of special backward-iteration sequence}

It was shown in the previous section that any backward-iteration sequence with bounded hyperbolic step tend to a BRFP. Now we will show that any isolated BRFP is a limit of a special backward-iteration sequence. This special backward-iteration sequence will be a cornerstone in the construction of conjugation near BRFP.

We will follow the idea, similar to that in one-dimensional case outlined in \cite{PPC1}. Note that in one dimension BRFPs with multipliers bounded by the same constant have to be isolated, as it follows from theorem of Cowen and Pommerenke \cite{CP}. Here we will have to impose this as a hypothesis, since not all BRFPs are isolated in higher dimensions (see Example \ref{ex:quadpol}).

\begin{lemma}\label{lemma:mainlemma}
Let $f$ be an analytic self-map of $\B^N$ and $(1,0)$ be a BRFP for $f$ with multiplier $1<\alpha<\infty$, isolated from the other BRFP's with multipliers less or equal to $\alpha$. Then there exist a backward-iteration sequence $\{Z_n\}_{n=0}^{\infty}$ tending to $(1,0)$ such that 
\[
d(Z_n,Z_{n+1})\leq a=\displaystyle\frac{\alpha-1}{\alpha+1}.
\]
\end{lemma}

In this and the following sections we will need a geometric notion slightly different from Koranyi regions:

\begin{definition}\label{def:restrictedcurve}
For $X\in\partial\B^N$, a curve $\sigma:[0,1)\to\B^N$ such that $\sigma(t)\to X$ as $t\to 1$ is called special if 
\begin{align}
\lim_{t\to 1}\frac{\|\sigma(t)-\sigma_X(t)\|^2}{1-\|\sigma_X(t)\|^2}=0,\label{specialcond}
\end{align}
and restricted if it is special and its orthogonal projection $\sigma_X:=(\sigma,X)X$ is non-tangential.
\end{definition}
\begin{definition}\label{def:restrKlimit}
We will say that $f:\B^N\to\B^N$ has restricted $K$-limit $Y$ at $X\in\partial\B^N$ if $f(\sigma(t))\to Y$ as $t\to 1$ for any restricted curve $\sigma$.
\end{definition}
\begin{remark}
Restricted $K$-limit is a weaker notion than $K$-limit: a function having $K$-limit has restricted $K$-limit, and a function having restricted $K$-limit has non-tangential limit, see \cite{Abate}.
\end{remark}
We will need the following result on the behavior of the radial and tangential components of $f$ near the BRFP $(1,0)$:
\begin{lemma}\label{lemma:angular}
Let $f:\B^N\to\B^N$ be analytic and $(1,0)$ be a fixed point for $f$ with multiplier $\alpha$ (in the sense of Julia's lemma). Then the following functions are bounded in every Koranyi region:
\begin{enumerate}
\item{$\displaystyle\frac{1-\pi_1(f(Z))}{1-\pi_1(Z)}$,}
\item{$\displaystyle\frac{f(Z)-\pi_1(f(Z))(1,0)}{\left|1-\pi_1(Z)\right|^{1/2}}$,}
\end{enumerate}
where $\pi_1(Z)=(Z,(1,0))$. Moreover, the function (1) has restricted $K$-limit $\alpha$ at $(1,0)$, and the function (2) has restricted $K$-limit $0$ at $(1,0)$.
\end{lemma}
\begin{proof}
Apply theorem 2.2.29 (i) and (ii) from \cite{Abate} to the boundary fixed point $(1,0)$.
\end{proof}
\begin{proof}[Proof of Lemma \ref{lemma:mainlemma}] Let $D$ be a small enough (Euclidean) closed ball centered at $(1,0)$ that does not contain the Denjoy-Wolff point of $f$ or any other BRFP of $f$. Let $a_k=(\alpha^k-1)/(\alpha^k+1)$ and 
\[
H(a_k)=\left\{Z\in\B^N:\frac{|1-(Z,e_1)|^2}{1-\|Z\|^2}\leq\frac{(1-a_k)^2}{1-a_k^2}=\alpha^{-k}\right\},
\]

i.e. a horosphere whose intersection with the 1-dimensional subspace generated by $e_1=(1,0)$ is a disk with diameter $[(a_k,0),(1,0)]$. Let $n_0$ be the smallest integer such that $H(a_{n_0})\subseteq D$ and $r_k=a_{{n_0}+k}$. (We will identify $r_k$ with $(r_k,0)\in\B^N$, that will cause no confusion). Also let $H_k=H(r_k)$, $J=\partial D\cap\B^N$ and $\gamma_n$ be the line segment connecting $r_k$ and $f(r_k)$.

For each $k$, the sequence $\{f_n(r_k)\}_n$ converges to the Denjoy-Wolff point of $f$, hence eventually leaves $D$. So there exists a smallest integer $n_k$ such that $f_{n_k}(\gamma_k)$ intersects $J$. By Julia's lemma (Theorem \ref{thm:JuliaN}), $f(H_{k+1})\subseteq H_k$, so $f_j(\gamma_k)$ cannot intersect $J$ for $j=1, 2, \ldots k-1$ and thus $n_k\geq k$.

\begin{claim1}
$d(r_k,f(r_k))\xrightarrow[k\to\infty]{}a$.
\end{claim1}
By Lemma \ref{lemma:angular}, 
\[
\lim_{k\to\infty}\frac{1-\pi_1(f(r_k))}{1-r_k}=\alpha,
\]
and by the definition of multiplier
\begin{align}
\liminf_{k\to\infty}\frac{1-\|f(r_k)\|}{1-r_k}\geq \alpha.\label{multineq}
\end{align}
By (\ref{dBN}), the pseudo-hyperbolic distance $d$ in $\B^N$ must satisfy the relation:
\[
1-d^2(r_k,f(r_k))=\frac{(1-r_k^2)(1-\|f(r_k)\|^2)}{|1-r_k\pi_1(f(r_k))|^2}=\frac{(1+r_k)(1+\|f(r_k)\|)\displaystyle\frac{1-\|f(r_k)\|}{1-r_k}}{\left|\displaystyle\frac{1-r_k\pi_1(f(r_k))}{1-r_k}\right|^2}.
\]
Now
\[
\frac{1-r_k\pi_1(f(r_k))}{1-r_k}=\frac{1-r_k+r_k-r_k\pi_1(f(r_k))}{1-r_k}=1+r_k\frac{1-\pi_1(f(r_k))}{1-r_k}\longrightarrow 1+\alpha,
\]
and so
\[
\liminf_{k\to\infty}\left(1-d^2(r_k,f(r_k))\right)\geq\frac{4A}{(1+\alpha)^2}
\]
or
\[
\limsup_{k\to\infty}d(r_k,f(r_k))\leq\frac{\alpha-1}{\alpha+1}=a.
\]
We will need the following inequality for $d_k:=d(r_k,f(r_k))$:
\begin{align}
\frac{1-\|f(r_k)\|}{1-r_k}\leq\frac{1+d_k}{1-r_kd_k}.\label{imageineq}
\end{align}

In fact, this is a partial case of more general inequality:
\begin{claim}\label{distanceineq}
For all $Z,W\in\B^N$ and $d:=d_{\B^N}(Z,W)$
\[
\frac{1-||W\|}{1-\|Z\|}\leq\frac{1+d}{1-d\|Z\|}
\]
\end{claim}
\begin{proof}
Let $\Delta$ be a closed hyperbolic ball centered at $Z$ of (pseudo-hyperbolic) radius $d=d_{\B^N}(Z,W)$. This is a Euclidean ellipsoid, centered at $\displaystyle\frac{1-d^2}{1-d^2\|Z\|^2}Z$ and a disk of radius $\displaystyle\frac{1-\|Z\|^2}{1-d^2\|Z\|^2}d$, when restricted to the subspace generated by $Z$. Thus the point, which is closest to the origin must be in the subspace generated by $Z$, and has modulus 
\[
\displaystyle\frac{1-d^2}{1-d^2\|Z\|^2}\|Z\|-\frac{1-\|Z\|^2}{1-d^2\|Z\|^2}d=\frac{(\|Z\|-d)(1+d\|Z\|)}{1-d^2\|Z\|^2}=\frac{\|Z\|-d}{1-d\|Z\|}.
\]
Since $W\in\Delta$,
\[
1-\|W\|\leq 1-\frac{\|Z\|-d}{1-d\|Z\|}=\frac{1+d}{1-d\|Z\|}(1-\|Z\|),
\]
\[
\frac{1-\|W\|}{1-\|Z\|}\leq\frac{1+d}{1-d\|Z\|}.
\]

\end{proof}
By taking $limsup$ of both sides of (\ref{imageineq}),
\[
\limsup_{k\to\infty}\frac{1-\|f(r_k)\|}{1-r_k}\leq\frac{1+a}{1-a}=\alpha,
\]

so this with (\ref{multineq}) shows that $\displaystyle\lim_{k\to\infty}\frac{1-\|f(r_k)\|}{1-r_k}=\alpha$ and $\displaystyle\lim_{k\to\infty} d(r_k,f(r_k))=a$.

The final steps in the construction are exactly the same as in proof of lemma 1.4 in \cite{PPC1}.
\end{proof}

\begin{lemma}\label{mainlemma}
If $\{Z_n\}_{n=1}^{\infty}$ is backward-iteration sequence, which tends to $e_1=(1,0)$ (BRFP with multiplier $\alpha>1$) and $d(Z_n,Z_{n+1})\leq a=\frac{\alpha-1}{\alpha+1}$, then its image in the Siegel domain must satisfy the following properties:
\begin{align}
\lim_{n\to\infty}\frac{\rea  z_n}{t_n}=1, \label{xlimit}
\end{align}
\begin{align}
\lim_{n\to\infty}\frac{\ima  z_n}{t_n}=0, \label{ylimit}
\end{align}
\begin{align}
\lim_{n\to\infty}\frac{\|w_n\|^2}{t_n}=0, \label{wlimit}
\end{align}
\begin{align}
\lim_{n\to\infty}\frac{t_n}{t_{n+1}}=\alpha, \label{tlimit}
\end{align}
where $t_n:=\rea  z_n-\|w_n\|^2$.
In particular, the sequence $\{Z_n\}$ is special, i.e.
\[
\lim_{n\to\infty}\frac{\|Z_n-(Z_n,e_1)e_1\|^2}{1-\|(Z_n,e_1)\|^2}=0.
\]
\end{lemma}
\begin{proof} By definition of multiplier 
\[
\liminf_{n\to\infty}\frac{1-\|Z_n\|}{1-\|Z_{n+1}\|}\geq \alpha=\frac{1+a}{1-a}.
\]
Applying Claim \ref{distanceineq} to $Z_n$, $Z_{n+1}$ and $r_n=d(Z_n,Z_{n+1})$, we have 
\[
\frac{1-\|Z_n\|}{1-\|Z_{n+1}\|}\leq\frac{1+r_n}{1-r_n\|Z_{n+1}\|}\leq\frac{1+a}{1-a\|Z_{n+1}\|}.
\]

Taking $\limsup$ of both sides, 
\[
\frac{1-\|Z_n\|}{1-\|Z_{n+1}\|}\to \alpha
\]
or, in Siegel domain,
\[
\frac{t_n}{t_{n+1}}\to \alpha, 
\]
so (\ref{tlimit}) is proved. Here we are going to use slightly different version of Cayley transform:
\[
\cC^{-1}(z,w):=\left(\frac{1-z}{1+z},\frac{2w}{1+z}\right),
\]
so that BRFP $(1,0)$ will be mapped to $\cC(1,0)=(0,0)$.

Consider the images of two consecutive points $Z_n$ and $Z_{n+1}$ under the automorphism 
$h_n:(z,w):=(z-i \ima  z_n+\|w_n\|^2-2(w,w_n),w-w_n)$, s.t. $h_n(Z_n)=(t_n,0)$ and denote $(\tilde{z}_n,\tilde{w}_n):=h_n(Z_{n+1})$. $h_n$ does not change the pseudo-hyperbolic distance in $\Hh^N$, so $d\left((t_n,0),(\tilde{z}_n,\tilde{w}_n)\right)=d(Z_n,Z_{n+1})\leq a$, which is 
\[
\|\tilde{z}_n-t_n\|^2+4t_n\|\tilde{w}_n\|^2\leq a^2\|\tilde{z}_n+t_n\|^2,
\]
\[
\|\tilde{z}_n-t_n\|^2+4t_n(\rea  \tilde{z}_n-t_{n+1})\leq a^2\|\tilde{z}_n+t_n\|^2,
\]
\[
(1-a^2)\|\tilde{z}_n+t_n\|^2\leq 4t_nt_{n+1},
\]
\[
\left|\frac{\tilde{z}_n}{t_n}+1\right|^2\leq\frac{4t_{n+1}}{t_n(1-a^2)}.
\]
Taking limsup of both sides and using (\ref{tlimit}),
\[
\limsup_{n\to\infty}\left|\frac{\tilde{z}_n}{t_n}+1\right|^2=\limsup_{n\to\infty}\left(\left|\frac{\rea \tilde{z}_n}{t_n}+1\right|^2+\left|\frac{\ima \tilde{z}_n}{t_n}\right|^2\right)\leq\left(1+\frac{1}{\alpha}\right)^2.
\]

Since $\rea \tilde{z}_n=t_{n+1}+\|\tilde{w}_n\|^2\geq t_{n+1}$,
\[
\limsup_{n\to\infty}\left(\left|\frac{t_{n+1}}{t_n}+1\right|^2+\left|\frac{\ima \tilde{z}_n}{t_n}\right|^2\right)\leq\left(1+\frac{1}{\alpha}\right)^2,
\]
\[
\left(\frac{1}{\alpha}+1\right)^2+\limsup_{n\to\infty}\left|\frac{\ima \tilde{z}_n}{t_n}\right|^2\leq\left(1+\frac{1}{\alpha}\right)^2.
\]

So,
\begin{align}
\frac{\ima \tilde{z}_n}{t_n}\to 0, \label{yshiftlimit}
\end{align}
which implies
\begin{align}
\frac{\rea \tilde{z}_n}{t_n}\to \frac{1}{\alpha} \label{xshiftlimit}
\end{align}
and
\begin{align}
\frac{\|\tilde{w}_n\|^2}{t_n}=\frac{\rea \tilde{z}_n}{t_n}-\frac{t_{n+1}}{t_n}\to 0. \label{wshiftlimit}
\end{align}
Now $w_{n+1}=w_n+\tilde{w}_n$, $w_{n+k}=w_n+\displaystyle\sum_{j=0}^{k-1}\tilde{w}_{n+j}$ $\forall k\geq 1$. 
\[
\|w_{n+k}\|\geq\|w_n\|-\sum_{j=0}^{k-1}\|\tilde{w}_{n+j}\|,
\]
\[
0\geq\|w_n\|-\sum_{j=0}^{\infty}\|\tilde{w}_{n+j}\|,
\]
\[
\|w_n\|\leq\sum_{j=0}^{\infty}\|\tilde{w}_{n+j}\|.
\]

Since $\frac{t_n}{t_{n+1}}\to \alpha>1$, pick $\epsilon$ such that $\alpha-\epsilon>1$, then for large enough $n$ $t_{n+1}\leq\frac{t_n}{\alpha-\epsilon}$ and $t_{n+j}\leq\frac{t_n}{(\alpha-\epsilon)^j}$.

Now by (\ref{wshiftlimit}), $\forall\delta>0\ \exists N=N(\delta)$ s.t. $\|\tilde{w}_n\|\leq\delta\sqrt{t_{n}}$ for $n\geq N$

\[
\|w_n\|\leq\sum_{j=0}^{\infty}\delta\sqrt{t_{n+j}}\leq\delta\sum_{j=0}^{\infty}\frac{\sqrt{t_n}}{(\alpha-\epsilon)^{j/2}}=\delta S\sqrt{t_n},
\]
where $S$ is finite sum. So 
\[
\frac{\|w_n\|^2}{t_n}\to 0 
\]
and
\[
\frac{\rea  z_n}{t_n}=\frac{t_n+\|w_n\|^2}{t_n}\to 1.
\]

Similarly, because $\ima  z_{n+1}=\ima  z_n+\ima  \tilde{z}_n+2 \ima \left\langle\tilde{w}_n,w_n \right\rangle$, $\left|2 \ima \left\langle\tilde{w}_n,w_n \right\rangle\right|\leq 2\|\tilde{w}_n\|\|w_n\|$ and using (\ref{yshiftlimit}), (\ref{wshiftlimit}) and (\ref{wlimit}), 
\[
\frac{\ima  z_n}{t_n}\to 0.
\]

The condition (\ref{specialcond}) for $(z_n,w_n)\to(1,0)$ being special in $\B^N$ is
\[
\lim_{n\to\infty}\frac{\|w_n\|^2}{1-|z_n|^2}=0
\]
or, in $\Hh^N$
\[
\lim_{n\to\infty}\frac{\frac{4\|w_n\|^2}{|1+z_n|^2}}{1-\left|\frac{1-z_n}{1+z_n}\right|^2}=\lim_{n\to\infty}\frac{\|w_n\|^2}{\rea  z_n}=0.
\]
But
\[
\lim_{n\to\infty}\frac{\|w_n\|^2}{\rea  z_n}=\lim_{n\to\infty}\frac{\frac{\|w_n\|^2}{t_n}}{\frac{\rea  z_n}{t_n}}=0.
\]
\end{proof}

\end{section}\
\begin{section}{\bf Conjugation at boundary repelling fixed point}
The aim of this section is to solve equation (\ref{conjugation})
in $\B^N$, where $\eta$ is an automorphism of $\B^N$ with the same dilatation coefficient at BRFP as $f$ and $\psi:\B^N\to\B^N$ is an analytic map with some regularity at BRFP. As in \cite{PPC1}, the conjugating map will be obtained via the sequence of iterates $f_n$ composed with appropriate automorphisms of $\B^N$. It will be convenient to build almost the entire construction in $\Hh^N$ with BRFP $0$.

We will start with several technical statements.

Using the backward-iteration sequence $(z_n,w_n)\to 0$ as in Lemma \ref{mainlemma} with $t_n=\rea z_n-\|w_n\|^2$, define a sequence of automorphisms $\tau_n$ of $\Hh^N$ as $\tau_n:=h_n^{-1}\circ\delta_n^{-1}$, where
\[
h_n(z,w)=(z+\|w_n\|^2-iy_n-2\left\langle w,w_n\right\rangle, w-w_n),
\]
\[
h_n^{-1}(z,w)=(z+\|w_n\|^2+iy_n+2\left\langle w,w_n\right\rangle, w+w_n),
\]
\[
\delta_n(z,w)=(\frac{z}{t_n},\frac{w}{\sqrt{t_n}}),
\]
\[
\delta_n^{-1}(z,w)=(t_n z,\sqrt{t_n}w).
\]

Then $\tau_n(1,0)=(z_n,w_n)$.

\begin{lemma}\label{taulemma}
Let $\eta_k(z,w):=(\alpha^k z,\alpha^{k/2} w)$ and $\tau_n$ be defined as above. Then
\begin{enumerate}
	\item 
	 $\tau_{n+k}^{-1}\circ\tau_n\to\eta_k$, uniformly on compact subsets of $\Hh^N$, as $n$ tends to infinity,
	\item
   $\tau_{n+1}^{-1}\circ\eta^{-1}\circ\tau_n(z,w)\to(z,w)$, uniformly on compact sets of $\Hh^N$, as $n$ tends to infinity.
\end{enumerate}
\end{lemma}
\begin{proof}
Using definition of $\tau_n$ and properties (\ref{xlimit}), (\ref{ylimit}), (\ref{wlimit}) and (\ref{tlimit}),
\[
\tau_{n+k}^{-1}\circ\tau_n(z,w)=\delta_{n+k}\circ h_{n+k}\circ h_n^{-1}\circ\delta_n^{-1}(z,w)=
\]
\[
\left(\frac{t_n}{t_{n+k}}z+\frac{\|w_n\|^2}{t_{n+k}}+i\frac{y_n}{t_{n+k}}+2\frac{\sqrt{t_n}}{t_{n+k}}
\left\langle w,w_n\right\rangle+\frac{\|w_{n+k}\|^2}{t_{n+k}}-i\frac{y_{n+k}}{t_{n+k}}
-\frac{2}{t_{n+k}}\left\langle \sqrt{t_n}w+w_n,w_{n+k}\right\rangle, \right.
\]
\[
\left. \frac{\sqrt{t_n}w+w_n-w_{n+k}}{\sqrt{t_{n+k}}}\right)\xrightarrow[n\to\infty]{}\left(\alpha^k z,\alpha^{k/2}w\right)=\eta_k(z,w).
\]

\[
\tau_{n+1}^{-1}\circ\eta^{-1}\circ\tau_n(z,w)=\delta_{n+1}\circ h_{n+1}\circ\eta^{-1}\circ h_n^{-1}\circ\delta_n^{-1}(z,w)=
\]
\[
\left(\frac{t_n}{t_{n+1}\alpha}z+\frac{\|w_n\|^2}{t_{n+1}\alpha}+i\frac{y_n}{t_{n+1}\alpha}+2\frac{\sqrt{t_n}}{t_{n+1}\alpha}
\left\langle w,w_n\right\rangle+\frac{\|w_{n+1}\|^2}{t_{n+1}}-i\frac{y_{n+1}}{t_{n+1}}
-\frac{2}{t_{n+1}}\left\langle \sqrt{t_n}w+w_n,w_{n+1}\right\rangle, \right.
\]
\[
\left. \frac{\sqrt{t_n}w+w_n}{\sqrt{t_{n+1}}\sqrt{\alpha}}-\frac{w_{n+1}}{\sqrt{t_{n+1}}}\right)\xrightarrow[n\to\infty]{}\left(z,w\right).
\]
\end{proof}

\begin{claim}
$\tau_n(z,w)\xrightarrow[n\to\infty]{}0$ and stays in Koranyi region uniformly on compact sets of $\Hh^N$.
\end{claim}
\begin{proof}
\[
\tau_n(z,w)=\left(t_n z+\|w_n\|^2+iy_n+2\left\langle \sqrt{t_n}w,w_n\right\rangle,\sqrt{t_n}w+w_n\right).
\]

Condition for $(z,w)$ being in Koranyi region with vertex $0$ in $\Hh^N$:

\[
\frac{|z|}{\rea  z-\|w\|^2}<M.
\]

For $\tau(z,w)$:

\[
\frac{\left|t_n z+\|w_n\|^2+iy_n+2\left\langle \sqrt{t_n}w,w_n\right\rangle\right|}{t_n \rea  z+\|w_n\|^2+
2\sqrt{t_n}\rea \left\langle w,w_n\right\rangle-\|\sqrt{t_n}w+w_n\|^2}
\]

\[
=\frac{\left|z+\frac{\|w_n\|^2}{t_n}+\frac{iy_n}{t_n}+2\left\langle w,\frac{w_n}{\sqrt{t_n}}\right\rangle\right|}{\rea  z+\frac{\|w_n\|^2}{t_n}+
2\rea \left\langle w,\frac{w_n}{\sqrt{t_n}}\right\rangle-\|w+\frac{w_n}{\sqrt{t_n}}\|^2}\xrightarrow[n\to\infty]{}\frac{|z|}{\rea  z-\|w\|^2}.
\]

The limit is bounded on compact subsets of $\Hh^N$, so $\tau_n(z,w)$ belong to some Koranyi region.
\end{proof}

\begin{claim}
Let $\phi:=f\circ\eta^{-1}$ in $\B^N$. Then
\[
\liminf_{z\to (1,0)}\frac{1-\|\phi(z)\|}{1-\|z\|}=1
\]
and Lemma \ref{lemma:angular} is applicable.
\end{claim}
\begin{proof}
\begin{align*}
\liminf_{z\to (1,0)}\frac{1-\|\phi(z)\|}{1-\|z\|}&=\liminf_{z\to (1,0)}\frac{1-\|f\circ\eta^{-1}(z)\|}{1-\|\eta^{-1}(z)\|}\lim_{z\to (1,0)}\frac{1-\|\eta^{-1}(z)\|}{1-\|z\|}\\
&=\liminf_{z\to (1,0)}\frac{1-\|f(z)\|}{1-\|z\|}\lim_{z\to (1,0)}\frac{1-\|\eta^{-1}(z)\|}{1-\|z\|}=\alpha\cdot\frac{1}{\alpha}=1.
\end{align*}

Since $\eta^{-1}$ is an automorphism that fixes $(1,0)$ and

\[
\lim_{z\to (1,0)}\frac{1-\|\eta^{-1}(z)\|}{1-\|z\|}=\lim_{z\to (1,0)}\frac{1-\|\eta^{-1}(z)\|^2}{1-\|z\|^2}=\lim_{(z,w)\to (0,0)}\frac{1-\|\cC^{-1}(\frac{z}{\alpha},\frac{w}{\sqrt{\alpha}})\|^2}{1-\|\cC^{-1}(z,w)\|^2}
\]

\[
=\lim_{(z,w)\to (0,0)}\frac{1-\left|\frac{1-z/\alpha}{1+z/\alpha}\right|^2-\frac{4\|w\|^2}{\alpha|1+z/\alpha|^2}} {1-\left|\frac{1-z}{1+z}\right|^2-\frac{4\|w\|^2}{|1+z|^2}}=\lim_{(z,w)\to (0,0)}\frac{
\frac{\rea  z-\|w\|^2}{\alpha}}{\rea  z-\|w\|^2}\cdot\frac{|1+z|^2}{\left|1+\frac{z}{\alpha}\right|}=\frac{1}{\alpha}.
\]
\end{proof}

Now consider a normal family $\left\{f_n\circ\tau_n\circ p_1\right\}$, where $p_1(z,w)=(z,0)$.

\begin{claim} The sequence $\tau_n\circ p_1(z,w)\to 0$ is restricted uniformly on compact subsets of $\Hh^N$.
\end{claim}

\begin{proof} Note that $\tau_n\circ p_1(z,w)=(t_n z+\|w_n\|^2+iy_n,w_n)$.

Following Definition \ref{def:restrictedcurve}, we need to show that $\tau_n\circ p_1(z,w)$ is special in $\Hh^N$:
\[
\lim_{n\to\infty}\frac{\|w_n\|^2}{\rea (t_n z+\|w_n\|^2+iy_n)}=\lim_{n\to\infty}\frac{\frac{\|w_n\|^2}{t_n}}{\rea  z+\frac{\|w_n\|^2}{t_n}}=0,
\]
and that the projection on the first component is non-tangential, i.e that
\[
\frac{\left|t_n z+\|w_n\|^2+iy_n\right|}{\rea (t_n z+\|w_n\|^2+iy_n)}
\]
is bounded above, but
\[
\lim_{n\to\infty}\frac{\left|t_n z+\|w_n\|^2+iy_n\right|}{\rea (t_n z+\|w_n\|^2+iy_n)}=\lim_{n\to\infty}\frac{\left|z+\frac{\|w_n\|^2}{t_n}+i\frac{y_n}{t_n}\right|}{\rea z+\frac{\|w_n\|^2}{t_n}}=\frac{|z|}{\rea  z},
\]
so it is bounded ucss of $\Hh^N$.
\end{proof}

Thus Lemma \ref{lemma:angular} is applicable to the function $\phi=f\circ\eta^{-1}$ and the sequence $\tau_n\circ p_1(z,w)$, which gives us the following 
\begin{lemma}\label{dlemma}
\[
\lim_{n\to\infty}d\left(\tau_n(p_1(z,w)),\phi(\tau_n(p_1(z,w)))\right)=0.
\]
\end{lemma}

\begin{proof}
Denote $(u_n,v_n):=\tau_n(z,0)$ and $(\tilde{u}_n,\tilde{v}_n):=\phi(\tau_n(z,0))$. Then the restricted $K$-limits (1) and (2) in Lemma \ref{lemma:angular} in when 
translated to $\Hh^N$ become

\[
\lim_{n\to\infty}\frac{\tilde{u}_n}{u_n}=1\ \ \ \ \mbox{and}\ \ \ \  \lim_{n\to\infty}\frac{\|\tilde{v}_n\|^2}{u_n}=0.
\]

Since $\displaystyle\lim_{n\to\infty}\frac{u_n}{t_n}=z$,

\[
\lim_{n\to\infty}\frac{\tilde{u}_n}{t_n}=z\ \ \ \ \mbox{and}\ \ \ \  \lim_{n\to\infty}\frac{\|\tilde{v}_n\|^2}{t_n}=0.
\]

Now $\displaystyle d((u_n,v_n),(\tilde{u}_n,\tilde{v}_n))^2=1-\frac{4(\rea  u_n-\|v_n\|^2)(\rea  \tilde{u}_n-\|\tilde{v}_n\|^2)}{|\tilde{u}_n+\bar{u}_n-2\left\langle\tilde{v}_n,v_n\right\rangle|^2}$.

\begin{align*}
\lim_{n\to\infty}\frac{4(\rea  u_n-\|v_n\|^2)(\rea  \tilde{u}_n-\|\tilde{v}_n\|^2)}{\left|\tilde{u}_n+\bar{u}_n-2\left\langle\tilde{v}_n,v_n\right\rangle\right|^2}&=\lim_{n\to\infty}\frac{4(\rea  \frac{u_n}{t_n}-\frac{\|v_n\|^2}{t_n})(\rea  \frac{\tilde{u}_n}{t_n}-\frac{\|\tilde{v}_n\|^2}{t_n})}{\left|\frac{\tilde{u}_n}{t_n}+\frac{\bar{u}_n}{t_n}-2\left\langle\frac{\tilde{v}_n}{\sqrt{t_n}},\frac{v_n}{\sqrt{t_n}}\right\rangle\right|^2}\\
&=\frac{4(\rea z-0)(\rea  z-0)}{\left|z+\bar{z}+0\right|^2}=1,
\end{align*}

and 
\[
\lim_{n\to\infty}d(\tau_n(z,0),\phi(\tau_n(z,0)))=0.
\]
\end{proof}

\begin{proof}[Proof of Theorem \ref{thm:mainconjN}]
Consider the normal family $\left\{f_n\circ\tau_n\circ p_1\right\}$ and let $\psi$ be one of its normal limits. Then, by Schwarz's lemma
\begin{align}
d(f_n\circ\tau_n(z,0),f_{n+1}\circ\tau_{n+1}(z,0))\leq d(\tau_n(z,0),f\circ\tau_{n+1}(z,0))\notag\\ 
\leq d(\tau_n(z,0),f\circ\eta^{-1}\circ\tau_n(z,0))+d(\eta^{-1}\circ\tau_n(z,0),\tau_{n+1}(z,0)).\label{d}
\end{align}
The first summand in (\ref{d}) tends to zero by lemma \ref{dlemma}, and the second does by part (2) of lemma \ref{taulemma}, so 
\[
d(f_n\circ\tau_n(z,0),f_{n+1}\circ\tau_{n+1}(z,0))\to 0
\]
as $n$ tends to infinity. It follows that if a subsequence $\left\{f_{n_k}\circ\tau_{n_k}\circ p_1\right\}$ converges ucss of $\Hh^N$ to $\psi$, then so does $\left\{f_{n_{k}+1}\circ\tau_{n_{k}+1
}\circ p_1\right\}$. By construction
\[
f_{n_{k}+1}\circ\tau_{n_{k}+1}\circ p_1=f\circ f_{n_k}\circ\tau_{n_{k}+1}\circ p_1,
\]
where the left hand-side tends to $\psi$, and it is enough to show that $f_{n_k}\circ\tau_{n_k+1}\circ p_1\to\psi\circ\eta^{-1}$ to prove (\ref{conjugation}). Note that $\eta^{-1}$ and $p_1$ are linear functions with diagonal matrices and therefore commute, so $f_{n_k}\circ\tau_{n_k}\circ\eta^{-1}\circ p_1\to\psi\circ\eta^{-1}$ and it is enough to show that
\[
d\left(f_{n_k}\circ\tau_{n_k}\circ\eta^{-1}\circ p_1(Z),f_{n_k}\circ\tau_{n_k+1}\circ p_1(Z)\right)\to 0.
\]
Applying Schwarz's lemma again,
\[
d\left(f_{n_k}\circ\tau_{n_k}\circ\eta^{-1}\circ p_1(Z),f_{n_k}\circ\tau_{n_k+1}\circ p_1(Z)\right)\leq
d\left(\tau_{n_k}\circ\eta^{-1}(z,0),\tau_{n_k+1}(z,0)\right)
\]
\[
=d\left(\tau^{-1}_{n_k+1}\circ\tau_{n_k}\circ\eta^{-1}(z,0),(z,0)\right)\to 0
\]
by statement (1) of Lemma \ref{taulemma}, so we have
\[
\psi=f\circ\psi\circ\eta^{-1},
\]
which is equivalent to (\ref{conjugation}).

All we are left to show is that $\psi$ fixes $0$. Note that the image of $\left(\frac{\alpha^k-1}{\alpha^k+1},0\right)$ under the Cayley transform is $a_k=\left(\alpha^{-k},0\right)$ and that $p_1(a_k)=a_k$. Then by definition of the sequence $Z_n$ and $\tau_n$ and Schwarz's lemma
\[
d\left(f_n\circ\tau_n(a_k),Z_k\right)=d\left(f_n\circ\tau_n(a_k),f_n(Z_{n+k})\right)\leq d\left(a_k,\tau^{-1}_n\circ\tau_{n+k}(1,0)\right)
\]
\[
=d\left(\eta^{-1}_k(1,0),\tau^{-1}_n\circ\tau_{n+k}(1,0)\right)\to 0,
\]
for any $k=1,2,\ldots$ as $n$ tends to infinity, by (1) of lemma \ref{taulemma}. Thus we have
\[
\psi(a_k)=Z_k.
\]
Define the sequence
\begin{align}
g_n(Z):=\tau^{-1}_n\circ\psi\circ\eta^{-1}_n(Z).\label{gndefinition}
\end{align}
Then $g_n((1,0))=(1,0)$ and $g_n(a_1)=\tau^{-1}_n(\tau_{n+1}(1,0))\to\eta^{-1}(1,0)=a_1$, as n tends to infinity. Hence any normal limit of ${g_n}$ fixes $(1,0)$ and $a_1$, and, by Corollary (2.2.15) from \cite{Abate}, must fix the entire subspace, containing $(1,0)$ and $a_1$, i.e. the set $\left\{(z,0)\in\Hh^N\right\}$. Note that $\psi(z,w)=\psi(z,0)$ and by (\ref{gndefinition}) $g_n(z,w)=g_n(z,0)$, so $g_n\to p_1$.

Consider a straight line segment connecting $(1,0)$ and $(0,0)$. Obviously it is special curve and by theorem (2.2.25) from \cite{Abate} $\psi$ will have restricted $K$-limit $0$ at $0$ if
\begin{align}
\displaystyle\lim_{t\to 0}\psi(t,0)=0.\label{psilimit}
\end{align}
By (\ref{gndefinition}), $\psi=\tau_n\circ g_n\circ\eta_n$. Consider a straight line segment connecting $(\alpha^{-(n+1)},0)$ to $(\alpha^{-n},0)$. It will be mapped by $\eta_n$ to a segment $[(\alpha^{-1},0),(1,0)]$. Pick a point $(t,0)$ on this segment. Then
\begin{align*}
\left\|\tau_n(g_n(t,0))\right\|\leq\left\|\tau_n(g_n(t,0))-\tau_n(t,0)\right\|+\left\|\tau_n(t,0)\right\|\xrightarrow[n\to\infty]{}0,
\end{align*} 
since $g_n(t,0)\to (t,0)$, $\tau_n(t,0)\to 0$ uniformly in $t$ and $\tau_n'$ is bounded, and (\ref{psilimit}) follows.

Now we can show that $\left\{f_n\circ\tau_n\circ p_1\right\}$ actually converges to $\psi$. By Schwarz's lemma, (\ref{conjugation}) and (\ref{gndefinition})
\begin{align*}
&d\left(f_n\circ\tau_n\circ p_1(z,w),\psi(z,w)\right)=d\left(f_n\circ\tau_n\circ p_1(z,w),\psi\circ\eta_n\circ\eta^{-1}_n(z,w)\right)\\
= &d\left(f_n\circ\tau_n\circ p_1(z,w),f_n\circ\psi\circ\eta^{-1}_n(z,w)\right)\leq
d\left(\tau_n\circ p_1(z,w),\psi\circ\eta^{-1}_n(z,w)\right)\\
= &d\left(p_1(z,w),g_n(z,w)\right)\xrightarrow[n\to\infty]{}0.
\end{align*}
\end{proof}
\end{section}

\begin{section}{\bf Conjugation for expandable maps}

In this section we will provide conjugation for the maps with some regularity at the BRFP. This class of maps was introduced in \cite{BG}:

\begin{definition}\label{def:expandable}
Let $f:\Hh^N\to\Hh^N$ be holomorphic. We will call the map $f$ {\em expandable} at $0$ (write $f\in\cE^{1}_{\Hh^N}(0)$), if $f$ has the following expansion near $0$:
\[
f(z,w)=(\alpha z+o(|z|), Aw+o(|z|^{1/2})).
\]
In particular, $0$ is a fixed point of $f$.
\end{definition}
By applying part (1) of Lemma \ref{lemma:angular} to any special sequence $(z_n,w_n)\to 0$, we obtain
\[
\lim_{n\to\infty}\frac{\alpha z_n+o(|z_n|)}{z_n}=\alpha,
\]
i.e. $\alpha$ must be the dilation coefficient of $f$ at $0$.

\begin{remark}
Note that $A$ cannot have eigenvalues $|a_{j,j}|^2>\alpha$, because otherwise $f(\Hh^N)\not\subset \Hh^N$.
\end{remark}

\begin{proof}[Proof of Theorem \ref{thm:conjexpand}] The construction is essentially the same as in section 4. We modify the definition of $\tau_n$ as follows: $\tau_n:=\Omega^{-n}\circ h_n^{-1}\circ\delta_n^{-1}$, where $\Omega$ is as in the statement of Theorem \ref{thm:conjexpand}. The following two limits are generalization of lemma \ref{taulemma}:
\[
\tau_{n+k}^{-1}\circ\tau_n(z,w)=\delta_{n+k}\circ h_{n+k}\circ\Omega^{k}\circ h_n^{-1}\circ\delta_n^{-1}(z,w)=
\]
\[
\left(\frac{t_n}{t_{n+k}}z+\frac{\|w_n\|^2}{t_{n+k}}+i\frac{y_n}{t_{n+k}}+2\frac{\sqrt{t_n}}{t_{n+k}}
\left\langle w,w_n\right\rangle+\frac{\|w_{n+k}\|^2}{t_{n+k}}-i\frac{y_{n+k}}{t_{n+k}}
-\frac{2}{t_{n+k}}\left\langle \Omega^{k}(\sqrt{t_n}w+w_n),w_{n+k}\right\rangle, \right.
\]
\[
\left. \frac{\Omega^{k}(\sqrt{t_n}w+w_n)-w_{n+k}}{\sqrt{t_{n+k}}}\right)\xrightarrow[n\to\infty]{}\left({\alpha}^k z,\Omega^{k}{\alpha}^{k/2}w\right)=:\eta_k(z,w).
\]

(Here $\eta_k$ differs from previous $\eta_k$ by rotation by $\Omega^{k}$.)

\[
\tau_{n+1}^{-1}\circ\eta^{-1}\circ\tau_n(z,w)=\delta_{n+1}\circ h_{n+1}\circ\Omega^{n+1}\circ\eta^{-1}\circ\Omega^{-n}\circ h_n^{-1}\circ\delta_n^{-1}(z,w)=
\]
\[
\left(\frac{t_n}{t_{n+1}{\alpha}}z+\frac{\|w_n\|^2}{t_{n+1}{\alpha}}+i\frac{y_n}{t_{n+1}{\alpha}}+2\frac{\sqrt{t_n}}{t_{n+1}{\alpha}}
\left\langle w,w_n\right\rangle+\frac{\|w_{n+1}\|^2}{t_{n+1}}-i\frac{y_{n+1}}{t_{n+1}}
-\frac{2}{t_{n+1}}\left\langle \sqrt{t_n}w+w_n,w_{n+1}\right\rangle, \right.
\]
\[
\left. \frac{\sqrt{t_n}w+w_n}{\sqrt{t_{n+1}}\sqrt{\alpha}}-\frac{w_{n+1}}{\sqrt{t_{n+1}}}\right)\xrightarrow[n\to\infty]{}\left(z,w\right).
\]

Now $\phi(z,w):=f\circ\eta^{-1}(z,w)=f({\alpha}^{-1}z,\Omega^{-1}{\alpha}^{-1/2}w)=(z+o(|z|),\frac{\Omega^{-1}A}{\sqrt{\alpha}}w+o(|z|^{1/2}))$.

Let $p_L(z,w)=(z,w_1,\ldots,w_L,0,\ldots,0)$, i.e. projection on the first $1+L$ dimensions.

Denote $(u_n,v_n):=\tau_n(p_L(z,w))$ and $(\tilde{u}_n,\tilde{v}_n):=\phi(\tau_n(p_L(z,w)))$. Then $u_n=t_n z+\|w_n\|^2+iy_n+2\left\langle \sqrt{t_n}p_L(w),w_n\right\rangle$ and $v_n=\Omega^{-n}(\sqrt{t_n}p_L(w)+w_n)$. Since

\[
\lim_{n\to\infty}\frac{u_n}{t_n}=\lim_{n\to\infty}\frac{t_n z+\|w_n\|^2+iy_n+2\left\langle \sqrt{t_n}p_L(w),w_n\right\rangle}{t_n}=z,
\]

$o(|u_n|)=o(t_n)$ and $o(|u_n|^{1/2})=o(\sqrt{t_n})$, and, consequently,
$\tilde{u}_n=u_n+o(t_n)$ and

\[ \tilde{v}_n=\frac{\Omega^{-1}A}{\sqrt{\alpha}}v_n+o(\sqrt{t_n})=\frac{\Omega^{-(n+1)}A\sqrt{t_n}}{\sqrt{\alpha}}p_L(w)+\frac{\Omega^{-(n+1)}A}{\sqrt{\alpha}}w_n+o\sqrt{t_n})=\Omega^{-n}\sqrt{t_n}p_L(w)+o(\sqrt{t_n}).
\]

The pseudo-hyperbolic distance in $\Hh^N$ is 

\[
 d^2((u_n,v_n),(\tilde{u}_n,\tilde{v}_n))=1-\frac{4(\rea  u_n-\|v_n\|^2)(\rea  \tilde{u}_n-\|\tilde{v}_n\|^2)}{|\tilde{u}_n+\bar{u}_n-2\left\langle\tilde{v}_n,v_n\right\rangle|^2},
\]
and because
\begin{align*}
&\lim_{n\to\infty}\frac{4(\rea  u_n-\|v_n\|^2)(\rea  \tilde{u}_n-\|\tilde{v}_n\|^2)}{\left|\tilde{u}_n+\bar{u}_n-2\left\langle\tilde{v}_n,v_n\right\rangle\right|^2}\\
=&\lim_{n\to\infty}\frac{(\rea \frac{u_n}{t_n}-\frac{\|v_n\|^2}{t_n})(\rea \frac{u_n}{t_n}+\frac{o(t_n)}{t_n}-\|\Omega^{-n}p_L(w)+\frac{o(\sqrt{t_n})}{\sqrt{t_n}}\|^2)}{\left|\rea \frac{u_n}{t_n}+\frac{o(t_n)}{t_n}-\left\langle\Omega^{-n}p_L(w)+\frac{o(\sqrt{t_n})}{\sqrt{t_n}},\frac{v_n}{\sqrt{t_n}}\right\rangle\right|^2}\\
=&\frac{(\rea  z -\|p_L(w)\|^2)(\rea  z -\|p_L(w)\|^2)}{\left|\rea  z-\left\langle \Omega^{-n}p_L(w),\Omega^{-n}p_L(w)\right\rangle\right|^2}=1,
\end{align*}

$d^2((u_n,v_n),(\tilde{u}_n,\tilde{v}_n))\to 0$, i.e. conclusion analogous to the statement of lemma \ref{dlemma} holds.

Now define $\psi$ as one of the normal limits of $\left\{f_n\circ\tau_n\circ p_L\right\}$. The above computations shows that if $f_{n_k}\circ\tau_{n_k}\circ p_L$ converges to $\psi$, then $f_{n_k+1}\circ\tau_{n_k+1}\circ p_L$ also converges to $\psi$.
It is enough to show that $f_{n_k}\circ\tau_{n_k+1}\circ p_L$ converges to $\psi\circ\eta^{-1}$ uniformly on compact subsets of $\Hh^N$. Note that $\eta^{-1}\circ p_L=p_L\circ\eta^{-1}$. Because
\begin{align*}
d\left(f_{n_k}\circ\tau_{n_k}\circ\eta^{-1}\circ p_L(z,w),f_{n_k}\circ\tau_{n_k+1}\circ p_L(z,w)\right)\\
=d\left(\tau_{n_k+1}^{-1}\circ \tau_{n_k}\circ\eta^{-1}\circ p_L(z,w),p_L(z,w)\right)\xrightarrow[n\to\infty]{}0,
\end{align*}
\[
\lim_{n\to\infty}f_{n_k}\circ\tau_{n_k+1}\circ p_L(z,w)=\lim_{n\to\infty}f_{n_k}\circ\tau_{n_k}\circ\eta^{-1}\circ p_L(z,w)=
\psi\circ\eta^{-1}(z,w),
\]
and (\ref{conjugation}) holds.

By the same reasoning as in proof of Theorem (\ref{thm:mainconjN}), $\psi$ fixes $0$ in the sense of restricted $K$-limits.
\end{proof}

\begin{remark} Note that in the case when eigenvalues of $A$ are equal to $\sqrt{\alpha}$, $f$ will be conjugated to same automorphism $\eta$ as in Theorem \ref{thm:mainconjN}, but the intertwining map $\psi$ will be different (its image needs not be one-dimensional).
\end{remark}

\begin{remark}
Consider the hyperbolic map $f:\Hh^N\to\Hh^N$ with the Denjoy-Wolff point infinity and BRFP $0$ with multiplier $1<\alpha<\infty$ : $f(z,w)=(\alpha z,0)$. Clearly, the image of $f$ is one-dimensional and from (\ref{conjugation}) we have that image of $\psi$ must be one-dimensional, so the result of Theorem \ref{thm:mainconjN} cannot be improved in general. For less trivial example, one may consider $f(z,w)=(\alpha z, \beta w)$ with $0<|\beta|^2<\alpha$. Now the image of $f$ has dimension $N$, but 
\[
\bigcap_{n=1}^{\infty}f_n(\Hh^N)
\]
is one-dimensional section of $\Hh^N$ and the range of the intertwining map $\psi$ is also one-dimensional.
\end{remark}
\end{section}

\begin{section}{\bf Examples and open questions}

\begin{subsection}{Examples}
In the beginning of this section we will describe all quadratic polynomials that map the two-dimensional Siegel domain $\Hh^2=\left\{(z,w)\in\C^2\left|\rea z>|w|^2\right.\right\}$ into itself while fixing zero, and completely characterize their dynamics. Some of these polynomials happen to have non-isolated BRFPs (see Example \ref{ex:quadpol}).

\begin{claim}
A quadratic polynomial $f:\C^2\to\C^2$ that fixes zero maps  $\Hh^2$ into $\Hh^2$ if and only if it is of the form $f(z,w)=(Az+Bw^2,Cw)$ with $A-|B|\geq |C|^2$.
\end{claim}
\begin{proof}
Consider the general form $f(z,w)=(f_1(z,w),f_2(z,w))=(az+bw+cz^2+dzw+ew^2,Az+Bw+Cz^2+Dzw+Ew^2)$. First we will show that most coefficients must be $0$.

Since $\rea f_1(z,w)>|f_2(z,w)|^2\geq 0$, then $\rea f_1(z,0)=\rea(az+cz^2)>0$ $\forall z$ such that $\rea z>0$. When $z\to 0$, $az+cz^2\sim az$, so $a>0$. Now $\rea f_1(z,0)=|z|(a\cos(\Arg z)+|c||z|\cos(2\Arg z+\Arg c))$, we can choose $\Arg z$ such that $\cos(2\Arg z+\Arg c)<0$ and $|z|$ large enough so $\rea f_1(z,0)<0$ unless $|c|=0$, so $c$ must be $0$. 

Thus $f(z,0)=(az,Az+Cz^2)$, and we must have $a|z|\cos(\Arg z)>|z|^2|A+Cz|^2$ or $a\cos(\Arg z)>|z||A+Cz|^2$. The right hand side goes to $\infty$ as $|z|\to\infty$ unless $C=A=0$.

Thus $f$ must be of the form $f(z,w)=(az+bw+dzw+ew^2,Bw+Dzw+Ew^2)$. Consider the set $\left\{(t,1)\in\C^2\left|\right.t>1\right\}\subset\Hh^2$. $f(t,1)=(at+dt+b+e,B+E+Dt)$ and $\rea(at+dt+b+e)<\left|B+E+Dt\right|^2$ for large enough $t$ unless $D=0$.

Now consider the set $\left\{(t^2+\epsilon,t)\in\C^2\left|\right.t>0\right\}\subset\Hh^2$. $\rea f_1(t^2+\epsilon,t)\leq a(t^2+\epsilon)+|b|t+|d|(t^2+\epsilon)t+|e|t^2<|Bt+Et^2|^2$ for large enough $t$ unless $E=0$.

To show that $d=0$, consider $\left\{(z,w)\in\C^2\left|\right.z=t^{2+\epsilon},|w|=t, t>1 \right\}\subset\Hh^2$. Then on this set $\rea f_1(z,w)\leq at^{2+\epsilon}+|b|t+|e|t^2+|d|t^{3+\epsilon}\cos\left(\Arg{d}+\Arg{w}\right)$. We can choose $\Arg{w}$ such that $\cos\left(\Arg{d}+\Arg{w}\right)<0$ and $t$ large enough to make $\rea f_1(z,w)< 0$, unless $d=0$.

The last part is to show that $b=0$. Consider $\left\{(z,w)\in\C^2\left|\right.z=t^{2-\epsilon},|w|=t, 1>t>0 \right\}\subset\Hh^2$. Then on this set $\rea f_1(z,w)\leq at^{2-\epsilon}+|b|t\cos\left(\Arg{b}+\Arg{w}\right)+|e|t^2$. We can choose $\Arg{w}$ such that $\cos\left(\Arg{b}+\Arg{w}\right)<0$ and $t$ close enough to $0$ such that $\rea f_1(z,w)<0$ unless $b=0$.

Thus $f$ has only three nonzero terms, and (by changing notations) the function must have form $f(z,w)=(Az+Bw^2,Cw)$. If $A-|B|\geq |C|^2$, then $\rea(Az+Bw^2)\geq A\rea{z}-|B||w|^2>(A-|B|)|w|^2\geq |C|^2|w|^2$ on $\Hh^2$ and hence $f(\Hh^2)\subseteq\Hh^2$. If $A-|B|<|C|^2$, we can choose $\Arg w$ such that $\cos\left(\Arg{B}+2\Arg{w}\right)=-1$ and $\rea{z}=|w|^2+\frac{\epsilon}{A}|w|^2$, where $\epsilon=|C|^2-A+|B|>0$, and then  $\rea(Az+Bw^2)=A\rea{z}+|B||w|^2\cos\left(\Arg{B}+2\Arg{w}\right)=A\rea{z}-|B||w|^2=(A-|B|+\epsilon)|w|^2=|C|^2|w|^2$, thus $f(\Hh^2)\not\subseteq\Hh^2$.
\end{proof}

\begin{claim}

\begin{enumerate}
\item{Aside from the trivial cases $A=0$ (must be zero map, because then $B=C=0$) and $C=0$ (one-dimensional projection) $f(z,w)=(Az+Bw^2,Cw)$ has well-defined inverse on $\Hh^2$
\[
f^{-1}(z,w)=\left(\frac{z}{A}-\frac{B}{AC^2}w^2,\frac{w}{C}\right)
\]
(though its image may be outside of the Siegel domain).}
\item{$n^{th}$ iterate of $f$ has the form 
\[
f^{\circ n}(z,w)=\left(A^nz+\frac{A^n-C^{2n}}{A-C^2}Bw^2,C^nw\right).
\]}
\end{enumerate}
\end{claim}
\begin{proof}
(1) is obvious. (2) can be shown by induction.
\end{proof}

Now we will find fixed points and classify the dynamical behavior of polynomials based on them.

Cases $C=0$ (projection on the first dimension) and $B=0$ (linear map) are trivial. So assume $B\neq 0$ and $C\neq 0$. To find the set of finite fixed points (either inner or boundary) we need to solve

\[
\begin{cases}
Az+Bw^2=z\\
Cw=w
\end{cases}
\]

If $C=1$, we can assume $A>1$ (otherwise $B=0$ and the map is identity). Then there are solutions $\displaystyle\left(-\frac{Bw^2}{A-1},w\right)$. Since
\[
\rea\left(-\frac{Bw^2}{A-1}\right)-|w|^2\leq\frac{|B||w|^2}{A-1}-|w|^2=\frac{|B|+1-A}{A-1}|w|^2\leq 0,
\]
any solution must be on the boundary of $\Hh^2$ and nonzero solutions exist iff $A=|B|+1$. In this case, there are infinitely many fixed points on the boundary (see Example (\ref{ex:quadpol}) below).

If $C\neq 1$ then nonzero solutions exist iff $A=1$ and they have form $\displaystyle\left(z,0\right)$. Thus we have interior fixed points.

If $C\neq 1$ and $A\neq 1$ then there are no fixed points inside of the domain and only two fixed points on the boundary (zero and infinity). One of them must be the Denjoy-Wolff point and the other BRFP.

The dilatation coefficient at $(0,0)$ is 
\begin{align*}
c=\liminf_{(z,w)\to(0,0)}\frac{\rea(Az+Bw^2)-|C|^2|w|^2}{\rea z-|w|^2} &\geq\liminf_{(z,w)\to(0,0)}\frac{A\rea{z}-|B||w|^2-|C|^2|w|^2}{\rea z-|w|^2}\\
& \geq\liminf_{(z,w)\to(0,0)}\frac{A\rea{z}-A|w|^2}{\rea{z}-|w|^2}=A
\end{align*}
and value $A$ attained for $z=t\to 0$ and $w=0$, so $c=A$.

Thus if $A<1$ then zero is the Denjoy-Wolff point of $f$ and this is hyperbolic case $c=A<1$. If $A>1$ then $(0,0)$ is the BRFP with dilatation coefficient $A>1$ and infinity must be the Denjoy-Wolff point. The dilatation coefficient at infinity
\begin{align*}
c=\liminf_{(z,w)\to\infty}\frac{\rea(Az+Bw^2)-|C|^2|w|^2}{\rea z-|w|^2}\frac{|z+1|^2}{|Az+Bw^2+1|^2} &\leq\lim_{t\to\infty}\frac{A\rea{t}}{t}\frac{|t+1|^2}{|At+1|^2}=\frac{1}{A},
\end{align*}
thus $c\leq\frac{1}{A}<1$ and this is also hyperbolic case.

\begin{example}[Example of a quadratic function with non-isolated BRFP]\label{ex:quadpol}
Consider the function $f(z,w):=(2z+w^2,w)$. Then $f^{\circ n}(z,w)=(2^n z+(2^n-1)w^2,w)$, the Denjoy-Wolff point is infinity and this is the hyperbolic case. The curve $\{ (r^2,ir)|r\in\R\}$ is clearly the set of fixed points on the boundary. Any of those points can be mapped to $(0,0)$ by translation
\[
h_r(z,w):=(z+r^2+2irw,w-ir)
\]
with
\[
h_r^{-1}(z,w)=(z+r^2-2irw,w+ir)
\]
Then 
\[
h_r\circ f\circ h_r^{-1}(z,w)=h_r\circ f(z+r^2-2irw,w+ir)=h_r(2z+r^2-2irw+w^2,w+ir)=(2z+w^2,w),
\]
i.e. the behavior of the function in any of those points is the same as in $(0,0)$.

The dilatation coefficient at zero is
\[
c=\liminf_{(z,w)\to(0,0)}\frac{\rea(2z+w^2)-|w|^2}{\rea z-|w|^2}=1+\liminf_{(z,w)\to(0,0)}\frac{\rea{z}+\rea(w^2)}{\rea{z}-|w|^2}=2.
\]
Thus we have a set of BRFP's on the boundary with the same dilatation coefficient, neither of them is isolated.
\end{example}

\begin{remark}\label{rem:nonisol}
Though $(0,0)$ is non-isolated BRFP for $f(z,w):=(2z+w^2,w)$, the statement of Lemma (\ref{lemma:mainlemma}) still holds in this case. $Z_n=(\frac{1}{2^n},0)$ is clearly an example of backward-iteration sequence with step $d=\frac{1}{3}$. Consequently, it is still possible to construct a conjugation as in Theorem \ref{thm:mainconjN}. 
\end{remark}

Now we will describe another class of self-maps of $\Hh^2$, the construction of these will be based on a function of one-dimensional half-plane $\Hh$.

\begin{example}\label{ex:1to2dim}
Let $\phi:\Hh\to\Hh$ be a holomorphic function of right-hand side half-plane, of hyperbolic or parabolic type, with the Denjoy-Wolff point infinity. Define a function $f$ on $\Hh^2$ as $f(z,w):=(\phi(z-w^2)+w^2,w)$. This function is well-defined since $\forall (z,w)\in\Hh^2$ $\rea(z-w^2)\geq\rea z-|w|^2>0$. Moreover, by Julia's lemma in $\Hh$, $\rea\phi(z-w^2)\geq\rea(z-w^2)$ and thus $\rea(\phi(z-w^2)+w^2)\geq\rea z>|w|^2$, and the function $f$ maps $\Hh^2$ into itself.
\begin{claim}
Infinity is the Denjoy-Wolff point for $f$ and $f$ has the same type and same multiplier at infinity as $\phi$. Moreover, if $\phi$ has a BRFP $y_0 i\in\partial\Hh$ then $f$ has a 1-dimensional real submanifold $\{(y_0 i+t^2,t)|t\in\R\}$ of BRFPs.
\end{claim}
\begin{proof}
Iterates of $f$ have a form $f^{\circ n}(z,w)=(\phi^{\circ n}(z-w^2)+w^2,w)$ and clearly the Denjoy-Wolff point is infinity. Assume $\phi$ has multiplier $c_1\leq 1$ at infinity, then $f$ has multiplier
\begin{align*}
c=&\liminf_{(z,w)\to\infty}\frac{\rea(\phi(z-w^2)+w^2)-|w|^2}{\rea z-|w|^2}\left|\frac{z+1}{\phi(z-w^2)+w^2+1}\right|^2\\
\leq &\liminf_{z\to\infty}\frac{\rea \phi(z)}{\rea z}\left|\frac{z+1}{\phi(z)+1}\right|^2=c_1.
\end{align*}
Since $f(z,0)=(\phi(z),0)$ and using Julia's lemma (\ref{JuliaH}), we  have 
\[
\rea\phi(z)\geq\frac{1}{c}\rea z\hspace{0.3 in}or\hspace{0.3 in}\frac{\rea\phi(z)}{\rea z}\geq\frac{1}{c}\hspace{0.3 in}\forall z\in\Hh,
\]
and, taking limit of both sides,
\[
\frac{1}{c_1}\geq\frac{1}{c}.
\]
Thus $c_1=c$, the multipliers coincide and therefore functions $f$ and $\phi$ are of the same type (either both hyperbolic or both parabolic).

Now $f(y_0 i+t^2,t)=(\phi(y_0 i)+t^2,t)=(y_0 i+t^2,t)$ and $(y_0 i+t^2,t)$ is a BRFP for $f$ $\forall t\in\R$.
\end{proof}
\end{example}
\end{subsection}

\begin{subsection}{Open questions.}

\begin{subsubsection}{The dimension of the stable set.} The {\sf stable set} $\cS$ at the BRFP $q$ is defined as the union of all backward-iteration sequences with bounded pseudo-hyperbolic step that tend to $q$. In one dimension, $\cS=\psi(\Hh)$. It is important to understand the properties of the stable set in N dimensions, because it may help to find the "best possible" intertwining map, i.e. the intertwining map whose image has the largest dimension.
\end{subsubsection}

\begin{subsubsection}{Non-isolated fixed points and necessary conditions for conjugation at BRFP}
As we can see from Remark \ref{rem:nonisol}, the condition on the BRFP to be isolated is sufficient, but not necessary. It is still not known if there are any BRFP for which the conjugation construction does not work. One needs to prove a result, similar to Lemma \ref{lemma:mainlemma} for non-isolated BRFP or to find necessary conditions on BRFP so that the conjugation construction will work.
\end{subsubsection}

\begin{subsubsection}{Convergence of backward-iteration sequences in parabolic case.}
Theorem \ref{thm:main} generalizes the one-dimensional Theorem \ref{thm:main1dim} only in hyperbolic and elliptic cases. It is still not known whether backward-iteration sequences with bounded step always converge for parabolic maps in higher dimensions. 
\end{subsubsection}
\end{subsection}

\end{section}
\bibliography{backwardbib}
\end{document}